\documentclass[12pt,reqno]{amsart}
\usepackage{etex}
\usepackage{amscd,amssymb,amsmath,multicol,tikz,float}
\usepackage[left=1in,right=1in]{geometry}
\begin{document}

\restylefloat{table}
\newtheorem{thm}[equation]{Theorem}
\numberwithin{equation}{section}
\newtheorem{cor}[equation]{Corollary}
\newtheorem{expl}[equation]{Example}
\newtheorem{rmk}[equation]{Remark}
\newtheorem{conv}[equation]{Convention}
\newtheorem{claim}[equation]{Claim}
\newtheorem{lem}[equation]{Lemma}
\newtheorem{sublem}[equation]{Sublemma}
\newtheorem{conj}[equation]{Conjecture}
\newtheorem{defin}[equation]{Definition}
\newtheorem{diag}[equation]{Diagram}
\newtheorem{prop}[equation]{Proposition}
\newtheorem{notation}[equation]{Notation}
\newtheorem{tab}[equation]{Table}
\newtheorem{fig}[equation]{Figure}
\newcounter{bean}
\renewcommand{\theequation}{\thesection.\arabic{equation}}

\raggedbottom \voffset=-.7truein \hoffset=0truein \vsize=8truein
\hsize=6truein \textheight=8truein \textwidth=6truein
\baselineskip=18truept

\def\mapright#1{\ \smash{\mathop{\longrightarrow}\limits^{#1}}\ }
\def\mapleft#1{\smash{\mathop{\longleftarrow}\limits^{#1}}}
\def\mapup#1{\Big\uparrow\rlap{$\vcenter {\hbox {$#1$}}$}}
\def\mapdown#1{\Big\downarrow\rlap{$\vcenter {\hbox {$\ssize{#1}$}}$}}
\def\mapne#1{\nearrow\rlap{$\vcenter {\hbox {$#1$}}$}}
\def\mapse#1{\searrow\rlap{$\vcenter {\hbox {$\ssize{#1}$}}$}}
\def\mapr#1{\smash{\mathop{\rightarrow}\limits^{#1}}}
\def\ss{\smallskip}
\def\s{\sigma}
\def\l{\lambda}
\def\vp{v_1^{-1}\pi}
\def\at{{\widetilde\alpha}}

\def\sm{\wedge}
\def\la{\langle}
\def\ra{\rangle}
\def\ev{\text{ev}}
\def\od{\text{od}}
\def\on{\operatorname}
\def\ol#1{\overline{#1}{}}
\def\spin{\on{Spin}}
\def\cat{\on{cat}}
\def\Lbar{\overline{\Lambda}}
\def\qed{\quad\rule{8pt}{8pt}\bigskip}
\def\ssize{\scriptstyle}
\def\a{\alpha}
\def\bz{{\Bbb Z}}
\def\Rhat{\hat{R}}
\def\im{\on{im}}
\def\ct{\widetilde{C}}
\def\ext{\on{Ext}}
\def\sq{\on{Sq}}
\def\eps{\epsilon}
\def\ar#1{\stackrel {#1}{\rightarrow}}
\def\br{{\bold R}}
\def\bC{{\bold C}}
\def\bA{{\bold A}}
\def\bB{{\bold B}}
\def\bD{{\bold D}}
\def\bC{{\bold C}}
\def\bh{{\bold H}}
\def\bQ{{\bold Q}}
\def\bP{{\bold P}}
\def\bx{{\bold x}}
\def\bo{{\bold{bo}}}
\def\dh{\widehat{d}}
\def\si{\sigma}
\def\Vbar{{\overline V}}
\def\dbar{{\overline d}}
\def\wbar{{\overline w}}
\def\Sum{\sum}
\def\tfrac{\textstyle\frac}

\def\tb{\textstyle\binom}
\def\Si{\Sigma}
\def\w{\wedge}
\def\equ{\begin{equation}}
\def\b{\beta}
\def\G{\Gamma}
\def\L{\Lambda}
\def\g{\gamma}
\def\d{\delta}
\def\k{\kappa}
\def\psit{\widetilde{\Psi}}
\def\tht{\widetilde{\Theta}}
\def\psiu{{\underline{\Psi}}}
\def\thu{{\underline{\Theta}}}
\def\aee{A_{\text{ee}}}
\def\aeo{A_{\text{eo}}}
\def\aoo{A_{\text{oo}}}
\def\aoe{A_{\text{oe}}}
\def\vbar{{\overline v}}
\def\endeq{\end{equation}}
\def\sn{S^{2n+1}}
\def\zp{\bold Z_p}
\def\cR{{\mathcal R}}
\def\P{{\mathcal P}}
\def\cQ{{\mathcal Q}}
\def\cj{{\cal J}}
\def\zt{{\bold Z}_2}
\def\bs{{\bold s}}
\def\bof{{\bold f}}
\def\bq{{\bold Q}}
\def\be{{\bold e}}
\def\Hom{\on{Hom}}
\def\ker{\on{ker}}
\def\kot{\widetilde{KO}}
\def\coker{\on{coker}}
\def\da{\downarrow}
\def\colim{\operatornamewithlimits{colim}}
\def\zphat{\bz_2^\wedge}
\def\io{\iota}
\def\om{\omega}
\def\Prod{\prod}
\def\e{{\cal E}}
\def\zlt{\Z_{(2)}}
\def\exp{\on{exp}}
\def\abar{{\overline a}}
\def\xbar{{\overline x}}
\def\ybar{{\overline y}}
\def\zbar{{\overline z}}
\def\mbar{{\overline m}}
\def\nbar{{\overline n}}
\def\sbar{{\overline s}}
\def\kbar{{\overline k}}
\def\bbar{{\overline b}}
\def\et{{\widetilde E}}
\def\ni{\noindent}
\def\tsum{\textstyle \sum}
\def\coef{\on{coef}}
\def\den{\on{den}}
\def\lcm{\on{l.c.m.}}
\def\Ext{\operatorname{Ext}}
\def\iso{\approx}
\def\lra{\longrightarrow}
\def\vi{v_1^{-1}}
\def\ot{\otimes}
\def\psibar{{\overline\psi}}
\def\thbar{{\overline\theta}}
\def\mhat{{\hat m}}
\def\exc{\on{exc}}
\def\ms{\medskip}
\def\ehat{{\hat e}}
\def\etao{{\eta_{\text{od}}}}
\def\etae{{\eta_{\text{ev}}}}
\def\dirlim{\operatornamewithlimits{dirlim}}
\def\gt{\widetilde{L}}
\def\lt{\widetilde{\lambda}}
\def\st{\widetilde{s}}
\def\ft{\widetilde{f}}
\def\sgd{\on{sgd}}
\def\lfl{\lfloor}
\def\rfl{\rfloor}
\def\ord{\on{ord}}
\def\gd{{\on{gd}}}
\def\rk{{{\on{rk}}_2}}
\def\nbar{{\overline{n}}}
\def\MC{\on{MC}}
\def\lg{{\on{lg}}}
\def\cH{\mathcal{H}}
\def\cS{\mathcal{S}}
\def\cP{\mathcal{P}}
\def\N{{\Bbb N}}
\def\Z{{\Bbb Z}}
\def\Q{{\Bbb Q}}
\def\R{{\Bbb R}}
\def\C{{\Bbb C}}
\def\Lb{\overline\Lambda}
\def\mo{\on{mod}}
\def\xt{\times}
\def\notimm{\not\subseteq}
\def\Remark{\noindent{\it  Remark}}
\def\kut{\widetilde{KU}}
\def\Eb{\overline E}
\def\*#1{\mathbf{#1}}
\def\0{$\*0$}
\def\1{$\*1$}
\def\22{$(\*2,\*2)$}
\def\33{$(\*3,\*3)$}
\def\ss{\smallskip}
\def\ssum{\sum\limits}
\def\dsum{\displaystyle\sum}
\def\la{\langle}
\def\ra{\rangle}
\def\on{\operatorname}
\def\proj{\on{proj}}
\def\od{\text{od}}
\def\ev{\text{ev}}
\def\o{\on{o}}
\def\U{\on{U}}
\def\lg{\on{lg}}
\def\a{\alpha}
\def\bz{{\Bbb Z}}
\def\eps{\varepsilon}
\def\bc{{\bold C}}
\def\bN{{\bold N}}
\def\bB{{\bold B}}
\def\bW{{\bold W}}
\def\nut{\widetilde{\nu}}
\def\tfrac{\textstyle\frac}
\def\b{\beta}
\def\G{\Gamma}
\def\g{\gamma}
\def\zt{{\Bbb Z}_2}
\def\zth{{\bold Z}_2^\wedge}
\def\bs{{\bold s}}
\def\bx{{\bold x}}
\def\bof{{\bold f}}
\def\bq{{\bold Q}}
\def\be{{\bold e}}
\def\lline{\rule{.6in}{.6pt}}
\def\xb{{\overline x}}
\def\xbar{{\overline x}}
\def\ybar{{\overline y}}
\def\zbar{{\overline z}}
\def\ebar{{\overline \be}}
\def\nbar{{\overline n}}
\def\ubar{{\overline u}}
\def\bbar{{\overline b}}
\def\et{{\widetilde e}}
\def\M{\mathcal{M}}
\def\lf{\lfloor}
\def\rf{\rfloor}
\def\ni{\noindent}
\def\ms{\medskip}
\def\Dhat{{\widehat D}}
\def\what{{\widehat w}}
\def\Yhat{{\widehat Y}}
\def\abar{{\overline{a}}}
\def\minp{\min\nolimits'}
\def\sb{{$\ssize\bullet$}}
\def\mul{\on{mul}}
\def\N{{\Bbb N}}
\def\Z{{\Bbb Z}}
\def\S{\Sigma}
\def\Q{{\Bbb Q}}
\def\R{{\Bbb R}}
\def\C{{\Bbb C}}
\def\Xb{\overline{X}}
\def\eb{\overline{e}}
\def\notint{\cancel\cap}
\def\cS{\mathcal S}
\def\cR{\mathcal R}
\def\el{\ell}
\def\TC{\on{TC}}
\def\GC{\on{GC}}
\def\wgt{\on{wgt}}
\def\Ht{\widetilde{H}}
\def\wbar{\overline w}
\def\dstyle{\displaystyle}
\def\Sq{\on{sq}}
\def\Om{\Omega}
\def\ds{\dstyle}
\def\tz{tikzpicture}
\def\zcl{\on{zcl}}
\def\bd{\bold{d}}
\def\cM{\mathcal{M}}
\def\io{\iota}
\def\Vb#1{{\overline{V_{#1}}}}
\def\Ebar{\overline{E}}
\def\lb{\,\begin{picture}(-1,1)(1,-1)\circle*{4.5}\end{picture}\ }
\def\lbb{\,\begin{picture}(-1,1)(1,-1)\circle*{8}\end{picture}\ }
\def\zp{\Z_p}

\title
{The  connective $K$-theory of the Eilenberg-MacLane space $K(\zp,2)$}
\author{Donald M. Davis}
\address{Department of Mathematics, Lehigh University\\Bethlehem, PA 18015, USA}
\email{dmd1@lehigh.edu}
\author{W. Stephen Wilson}
\address{Department of Mathematics, Johns Hopkins University\\Baltimore, MD 01220, USA}
\email{wwilson3@jhu.edu}
\date{February 10, 2023}

\keywords{Adams spectral sequence, connective $K$-theory, Eilenberg-MacLane spaces}
\thanks {2000 {\it Mathematics Subject Classification}: 55T15 55N20, 55N15.}

\maketitle

\begin{abstract} We compute $ku^*(K(\zp,2))$ and $ku_*(K(\zp,2))$, the connective $KU$-cohomology and connective $KU$-homology groups of the mod-$p$ Eilenberg-MacLane space $K(\zp,2)$, using the Adams spectral sequence. We obtain a striking interaction between $h_0$-extensions and exotic extensions. The mod-$p$ connective $KU$-cohomology groups, computed elsewhere, are needed in order to establish higher differentials and exotic extensions in the integral groups.
\end{abstract}

\section{Main results}\label{intro}
In \cite{W} and \cite{DW}, the authors initiated a partial computation of the  connective $KU$-homology groups, $ku_*(K(\zt,2))$, of the mod-2 Eilenberg-MacLane space $K(\zt,2)$ in separate studies of Stiefel-Whitney classes of manifolds. We eventually turned to the associated cohomology groups, $ku^*(K(\zt,2))$, and were able to give a complete determination, via the Adams spectral sequence (ASS). This generalized nicely to the odd primes, and then we found a duality result (\cite{DD}) relating these homology and cohomology groups which enabled us to determine the homology groups $ku_*(K(\zp,2))$.

Let $K_2=K(\zp,2)$, with the prime $p$ implicit. We begin with a description of the $ku^*$-module $ku^*(K_2)$. Note that $ku^*=\Z_{(p)}[v]$ with $|v|=-2(p-1)$. We find that depiction via ASS charts is the most insightful way to envision the groups. There is a very nice interplay between extensions (multiplication by $p$) seen in Ext ($h_0$-extensions) and exotic extensions.
 We depict the ASS with cohomological (co)degrees increasing from right-to-left. We write $|x|=d$ if $x\in ku^d(K_2)$ or the associated $E_2$-term.

In $ku^*(K_2)$, there is a trivial submodule whose Poincar\'e series when $p=2$ is described at the end of Section \ref{E2sec}. It plays no role and  {\bf will be ignored from now on}. As a $ku^*$-module, $ku^*(K_2)$ is generated by certain products of elements of $E_2^{0}$, $y_0$, $y_i=y_0^{p^i}$, with $|y_i|=2p^i$, $z_j$ for $j\ge0$ with $|z_j|=2(p^{j+1}+1)$, and $q$ with $|q|=9$ if $p=2$, and $|q|=4p-1$ if $p$ is odd.

The even-graded part $ku^{\text{ev}}(K_2)$ is formed from shifted copies of $ku^*$-modules $A_k$ and $B_k$, which can be defined inductively as follows.
\begin{defin}\label{ABdef}
Let $k_0=1$ if $p$ is odd, and $k_0=2$ if $p=2$. Let $B_{k_0-1}=0$.
 Let $A_0=\langle z_0\rangle$ for all $p$. Inductively
$$B_k\text{ is built from }z_{k-1}^{p-1}B_{k-1},\ TP_{p^k-k}[v]z_k, \text{ and }y_{k-1}^{p-1}B_{k-1},\text{ if }k\ge k_0$$
and
$$A_k\text{ is built from }z_{k-1}^{p-1}B_{k-1},\ TP_{p^k}[v]z_k, \text{ and }y_{k-1}^{p-1}A_{k-1},\text{ if }k\ge1$$
with extensions determined by
\begin{equation}\label{extns}pz_k=vz_{k-1}^p\text{ for $k\ge2$, and }py_{k-1}^{p-1}z_{k-1}=v^{p^{k-1}(p-1)}z_{k}.\end{equation}
\end{defin}
Here $TP_i[v]$ is the truncated polynomial algebra over $\zp$ with generator $v$ and relation $v^i$.
 When we write something like $zB$, we mean that all elements of $B$ are multiplied by the element $z$.
Saying ``is built from'' means that these are successive quotients in a filtration as a $ku^*$-module. The extension formulas are only asserted up to multiplication by a unit is $\zp$, and can both occur on an element. For example, in Figure \ref{B7}, we have, in grading 116 when $p=2$, $2y_3z_3z_4=vy_3z_2^2z_4+v^8z_4^2$.

Figure \ref{B7} should enable the reader to envision $A_k$ and $B_k$ for $p=2$ and $k\le5$, and, by extrapolating, for all $k$. Elements connected by dashed lines are in $A_5$ but not in $B_5$. The long red lines, sometimes slightly curved, are the exotic extensions. The portion in gradings $\le102$, not including the top $v$-tower or the extensions to it, is $y_4A_4$ (or $y_4B_4$ if the dashed part is omitted). The portion in gradings $\ge106$, not including the $v$-tower on $z_5$ or the $h_0$-extensions from it, is $z_4B_4$. The reader is encouraged to understand how the case $k=5$ of Definition \ref{ABdef} is embodied in Figure \ref{B7}.

The portion in the lower right corner of Figure \ref{B7} in grading $\le 84$ and height $\le7$ is $y_3y_4A_3$, and $y_2y_3y_4A_2$ is in gradings $\le 74$. In Figure \ref{oddchart}, we present a schematic of $A_3$ and $B_3$ at the odd primes. Again the dashed portion is in $A_3$, but not $B_3$, and the triangle in the lower right portion is $y_1^{p-1}y_2^{p-1}A_1$.

 A generating set as a $\zp[v]$-module for $B_k$ is
\begin{equation}\label{Bk}\biggl\{z_j\prod_{i=j}^{k-1}\{z_i^{p-1},y_i^{p-1}\}:\ k_0\le j\le k\biggr\},\end{equation}
while $A_k$ has  additional generators
$$\begin{cases}z_1y_1\cdots y_{k-1}&p=2\\ z_0y_0^{p-1}\cdots y_{k-1}^{p-1}&\text{all }p.\end{cases}$$
The notation here means a product over all choices of one of the two elements in each factor. For example, $\prod_{i=1}^2\{z_i^{p-1},y_i^{p-1}\}=z_1^{p-1}z_2^{p-1}+z_1^{p-1}y_2^{p-1}+y_1^{p-1}z_2^{p-1}+y_1^{p-1}y_2^{p-1}$.
An empty product is defined to equal 1.

The following theorem explains how the portion of $ku^*(K_2)$ in even gradings is a direct sum of shifted versions of $A_k$ and $B_k$.
\begin{thm}\label{evthm}
Let $M_p[S]$ denote the set of monomials in the elements of a set $S$ raised to powers $<p$. Let
\begin{equation}\label{Mk}\M_k=(M_p[z_k,y_k]-\{z_k^{p-1},y_k^{p-1}\})\cdot M_p[z_i,y_i:\ i>k].\end{equation}
Let $\M_k^A$ be the set of monomials in $\M_k$ with no $z$-factors, and $\M_k^B=\M_k-\M_k^A$. Then
$$ku^{\text{ev}}(K_2)=\bigoplus_{k\ge1}\biggl(\bigoplus_{M\in\M_k^A}M\cdot A_k\oplus\bigoplus_{M\in \M_k^B}M\cdot B_k\biggr)$$
plus a trivial $ku^*$-module.\end{thm}
Note that the monomial 1 is in $\M_k^A$, so $A_k$ appears by itself, but $B_k$ does not. For example, if $p=2$, copies of $B_k$ appear multiplied by each monomial of the form
$$z_k^{\eps_k}y_k^{\d_k}z_{k+1}^{\eps_{k+1}}y_{k+1}^{\d_{k+1}}\cdots\text{ such that }\eps_k=\d_k\text{ and}\sum\eps_i\ge1.$$

\tikzset{
  testpic/.pic=
{\draw (0,0) -- (1,1) -- (1,0) -- (5,4) -- (5,2) -- (29,26);
\draw [dashed] (29,26) -- (34,28);
\draw [dashed] [color=red] (34,28) -- (34,0);
\draw (0,0) -- (34,0);
\draw (3,2) -- (3,0);
\draw (4,1) -- (4,3);
\draw [color=red] (5,0) to[out=98, in=262] (5,4);
\draw [color=red] (10,0) to[out=98, in=262] (10,8);
\draw [color=red] (11,1) to[out=98, in=262] (11,9);
\draw [color=red] (12,2) to[out=98, in=262] (12,10);
\draw [color=red] (13,3) to[out=98, in=262] (13,11);
\draw [color=red] (22,0) to[out=98, in=262] (22,4);
\draw [color=red] (22,3) to[out=83, in=270] (22,19);
\draw [color=red] (21,2) to[out=83, in=270] (21,18);
\draw [color=red] (20,1) to[out=83, in=270] (20,17);
\draw [color=red] (19,0) to[out=83, in=270] (19,16);
\draw [color=red] (27,0) to[out=83, in=270] (27,8);
\draw [color=red] (14,0) -- (14,4);
\draw [color=red] (23,4) -- (23,20);
\draw [color=red] (24,5) -- (24,21);
\draw [color=red] (25,6) -- (25,22);
\draw [color=red] (26,7) -- (26,23);
\draw [color=red] (27,8) -- (27,24);
\draw [color=red] (28,1) -- (28,25);
\draw [color=red] (29,2) -- (29,26);
\draw [color=red] (30,3) -- (30,11);
\draw [color=red] (31,0) -- (31,4);
\node at (0,0) {\lb};
\node at (1,0) {\lb};
\node at (1,1) {\lb};
\node at (2,1) {\lb};
\node at (3,2) {\lb};
\node at (4,3) {\lb};
\node at (5,4) {\lb};
\node at (2,0) {\lb};
\node at (3,1) {\lb};
\node at (4,2) {\lb};
\node at (5,3) {\lb};
\node at (6,4) {\lb};
\node at (7,5) {\lb};
\node at (8,6) {\lb};
\node at (9,7) {\lb};
\node at (10,8) {\lb};
\node at (11,9) {\lb};
\node at (12,10) {\lb};
\node at (13,11) {\lb};
\node at (3,0) {\lb};
\node at (4,1) {\lb};
\node at (5,2) {\lb};
\node at (6,3) {\lb};
\node at (7,4) {\lb};
\node at (8,5) {\lb};
\node at (9,6) {\lb};
\node at (10,7) {\lb};
\node at (11,8) {\lb};
\node at (12,9) {\lb};
\node at (13,10) {\lb};
\node at (14,11) {\lb};
\node at (15,12) {\lb};
\node at (16,13) {\lb};
\node at (17,14) {\lb};
\node at (18,15) {\lb};
\node at (19,16) {\lb};
\node at (20,17) {\lb};
\node at (21,18) {\lb};
\node at (22,19) {\lb};
\node at (23,20) {\lb};
\node at (24,21) {\lb};
\node at (25,22) {\lb};
\node at (26,23) {\lb};
\node at (27,24) {\lb};
\node at (28,25) {\lb};
\node at (29,26) {\lb};
\node at (30,26.4) {\lb};
\node at (31,26.8) {\lb};
\node at (32,27.2) {\lb};
\node at (33,27.6) {\lb};
\node at (34,28) {\lb};
\node at (5,0) {\lb};
\node at (6,1) {\lb};
\node at (9,0) {\lb};
\node at (10,1) {\lb};
\node at (10,0) {\lb};
\node at (11,1) {\lb};
\node at (12,2) {\lb};
\node at (13,3) {\lb};
\node at (14,4) {\lb};
\node at (14,0) {\lb};
\node at (15,1) {\lb};
\node at (17,0) {\lb};
\node at (18,1) {\lb};
\node at (18,0) {\lb};
\node at (19,1) {\lb};
\node at (20,2) {\lb};
\node at (21,3) {\lb};
\node at (22,4) {\lb};
\node at (19,0) {\lb};
\node at (20,1) {\lb};
\node at (21,2) {\lb};
\node at (22,3) {\lb};
\node at (23,4) {\lb};
\node at (24,5) {\lb};
\node at (25,6) {\lb};
\node at (26,7) {\lb};
\node at (27,8) {\lb};
\node at (28,9) {\lb};
\node at (29,10) {\lb};
\node at (30,11) {\lb};
\node at (22,0) {\lb};
\node at (23,1) {\lb};
\node at (26,0) {\lb};
\node at (27,1) {\lb};
\node at (27,0) {\lb};
\node at (28,1) {\lb};
\node at (29,2) {\lb};
\node at (30,3) {\lb};
\node at (31,4) {\lb};
\node at (31,0) {\lb};
\node at (32,1) {\lb};
\node at (33,2) {\lb};
\node at (33,0) {\lb};
\node at (34,0) {\lb};
\node at (34,1) {\lb};
\node at (34,3) {\lb};
\node at (32,5) {\lb};
\node at (33,6) {\lb};
\node at (34,7) {\lb};
\node at (31,12) {\lb};
\node at (32,13) {\lb};
\node at (33,14) {\lb};
\node at (34,15) {\lb};

\draw (17,0) -- (18,1) -- (18,0) -- (22,4);
\draw (19,0) -- (30,11);
\draw (19,0) -- (19,1) -- (20,2) -- (20,1) -- (21,2) -- (21,3) -- (22,4) -- (22,3);

\draw (2,1) -- (2,0) -- (13,11) -- (13,10);
\draw (6,3) -- (6,4) -- (7,5) -- (7,4) -- (8,5) -- (8,6) -- (9,7) -- (9,6) -- (10,7) -- (10,8) -- (11,9) -- (11,8) -- (12,9) -- (12,10);
\draw (3,0) -- (5,2);
\draw (9,0) -- (10,1) -- (10,0) -- (14,4);
\draw (5,0) -- (6,1);
\draw (14,0) -- (15,1);
\draw (22,0) -- (23,1);
\draw (26,0) -- (27,1) -- (27,0) -- (31,4);
\draw (31,0) -- (32,1);
\draw [dashed] [color=red] (33,0) -- (33,27.6);
\draw [dashed] [color=red] (32,1) -- (32,27.2);
\draw [dashed] [color=red] (31,4) -- (31,26.8);
\draw [dashed] [color=red] (30,11) -- (30,26.4);
\draw [dashed]  (33,0) -- (34,1);
\draw [dashed]  (32,1) -- (34,3);
\draw [dashed]  (31,4) -- (34,7);
\draw [dashed]   (30,11) -- (34,15);
\node at (0,-.4) {$136$};
\node at (3,-.4) {$130$};
\node at (0,-.9) {$z_2^2z_3z_4$};
\node at (3,-.9) {$z_5$};
\node at (2,-.9) {$z_4^2$};
\node at (5,-.4) {$126$};
\node at (5,-.9) {$y_2z_2z_3z_4$};
\node at (10,-.4) {$116$};
\node at (10,-.9) {$y_3z_3z_4$};
\node at (14,-.4) {$108$};
\node at (14,-.9) {$y_2y_3z_3z_4$};
\node at (17,-.4) {$102$};
\node at (17,-.9) {$y_4z_2^2z_3$};
\node at (19,-.4) {$98$};
\node at (19,-.9) {$y_4z_4$};
\node at (22,-.4) {$92$};
\node at (22,-.9) {$y_2y_4z_2z_3$};
\node at (27,-.9) {$y_3y_4z_3$};
\node at (27,-.4) {$82$};
\node at (31,-.4) {$74$};
\node at (31,-.9) {$y_2y_3y_4z_2$};
\node at (34,-.4) {$68$};
\node at (34,-.9) {$y_0y_1y_2y_3y_4z_0$};
}}

\begin{minipage}{6in}
\begin{fig}\label{B7}

{\bf $B_5$ and $A_5$ when $p=2$.}
\begin{center}
\begin{tikzpicture}
  \pic[rotate=90,scale=.56,transform shape] {testpic};
\end{tikzpicture}
\end{center}
\end{fig}
\end{minipage}

\vfill\eject

\tikzset{
  testpic4/.pic=

{\draw (-1,0) -- (70,0);
\draw (2,0) -- (66,48);
\draw (39,0) -- (67,21);
\draw (1,0) -- (29,21);
\draw (0,0) -- (12,9);
\draw (38,0) -- (50,9);
\draw (18,0) -- (30,9);
\draw (56,0) -- (68,9);
\draw [dotted] (12,9) -- (68,9);
\draw [dotted] (29,21) -- (67,21);
\draw [color=red] (51,9) -- (51,36.75);

\node at (66,48) {\lbb};
\node at (67,48.75) {\lbb};
\node at (68,49.5) {\lbb};
\node at (69,50.25) {\lbb};
\node at (67,21) {\lbb};
\node at (68,21.75) {\lbb};
\node at (69,22.5) {\lbb};
\node at (68,9) {\lbb};
\node at (69,9.75) {\lbb};
\draw (39,0) -- (39,.75);
\draw (41,1.5) -- (41,2.25);
\draw (43,3) -- (43,3.75);
\draw (45,4.5) -- (45,5.25);
\draw (47,6) -- (47,6.75);
\draw (49,7.5) -- (49,8.25);
\draw (50,8.25) -- (50,9);
\draw [dotted] (66,48) -- (69,50.25);
\draw [dotted] (67,21) -- (69,22.5);
\draw [dotted] (68,9) -- (69,9.75);
\draw [color=red] (56,0) -- (56,40.5);
\draw [color=red] (58,1.5) -- (58,42);
\draw [color=red] (60,3) -- (60,43.5);
\draw [color=red] (62,4.5) -- (62,45);
\draw [color=red] (64,6) -- (64,46.5);
\draw [color=red] (67,8.25) -- (67,21);
\draw [color=red] [dashed] (67,20.25) -- (67,48.75);
\draw [color=red] [dashed] (68,9) -- (68,49.5);
\draw [color=red] [dashed] (69,0) -- (69,50.25);
\draw [color=red] (66,7.5) -- (66,48);
\draw [color=red] (54,11.25) -- (54,39);
\draw [color=red] (52,9.75) -- (52,37.5);
\draw [color=red] (39,0) to[out=80, in=270] (39,27.75);
\draw [color=red] (41,1.5) to[out=80, in=270] (41,29.25);
\draw [color=red] (43,3) to[out=80, in=270] (43,30.75);
\draw [color=red] (45,4.5) to[out=80, in=270] (45,32.25);
\draw [color=red] (47,6) to[out=80, in=270] (47,33.75);
\draw [color=red] (50,8.25) to[out=85, in=270] (50,36);
\draw (2,0) -- (2,1.5);
\draw (1,0) -- (1,.75);
\draw (4,1.5) -- (4,3);
\draw (6,3) -- (6,4.5);
\draw (8,4.5) -- (8,6);
\draw (10,6) -- (10,7.5);
\draw (12,7.5) -- (12,9);
\draw (14,9) -- (14,9.75);
\draw (16,10.5) -- (16,11.25);
\draw (18,12) -- (18,12.75);
\draw (20,13.5) -- (20,14.25);
\draw (22,15) -- (22,15.75);
\draw (24,16.5) -- (24,17.25);
\draw (26,18) -- (26,18.75);
\draw (29,20.25) -- (29,21);
\draw [color=red] (18,0) to[out=98, in=262] (18,12.75);
\draw [color=red] (20,1.5) to[out=98, in=262] (20,14.25);
\draw [color=red] (22,3) to[out=98, in=262] (22,15.75);
\draw [color=red] (24,4.5) to[out=98, in=262] (24,17.25);
\draw [color=red] (26,6) to[out=98, in=262] (26,18.75);
\draw [color=red] (29,8.25) to[out=98, in=262] (29,21);
\draw [dotted] (18,12.75) -- (56,12.75);
\node [font=\fontsize{40}{0}] at (34,9) {$p-2$};
\node [font=\fontsize {40}{0}] at (34,12.75) {$p^2-p$};
\node [font=\fontsize {40}{0}] at (34,21) {$p^2-3$};
\node [font=\fontsize {40}{0}] at (64,50.25) {$p^3-1$};
\node [font=\fontsize {40}{0}] at (46,36.75) {$p^3-p^2+p-2$};
\draw [dotted] (49,36.75) -- (51,36.75);
\draw [dotted] (66,50.25) -- (69,50.25);
\node [font=\Huge] at (18,-1) {$y_1^{p-1}z_1z_2^{p-1}$};
\node [font=\Huge] at (56,-1) {$y_1^{p-1}y_2^{p-1}z_1$};
\node [font=\Huge] at (36,-1) {$y_2^{p-1}z_1^p$};
\node [font=\Huge] at (41,-1) {$y_2^{p-1}z_2$};
\node [font=\Huge] at (69,-1) {$y_0^{p-1}y_1^{p-1}y_2^{p-1}z_0$};
\node [font=\Huge] at (1,-1) {$z_2^p$};
\node [font=\Huge] at (2.5,-1) {$z_3$};
\node [font=\Huge] at (-1.5,-1) {$z_1^pz_2^{p-1}$};
\node at (69,0) {\lbb};
}}
\bigskip
\begin{minipage}{6in}
\begin{fig}\label{oddchart}

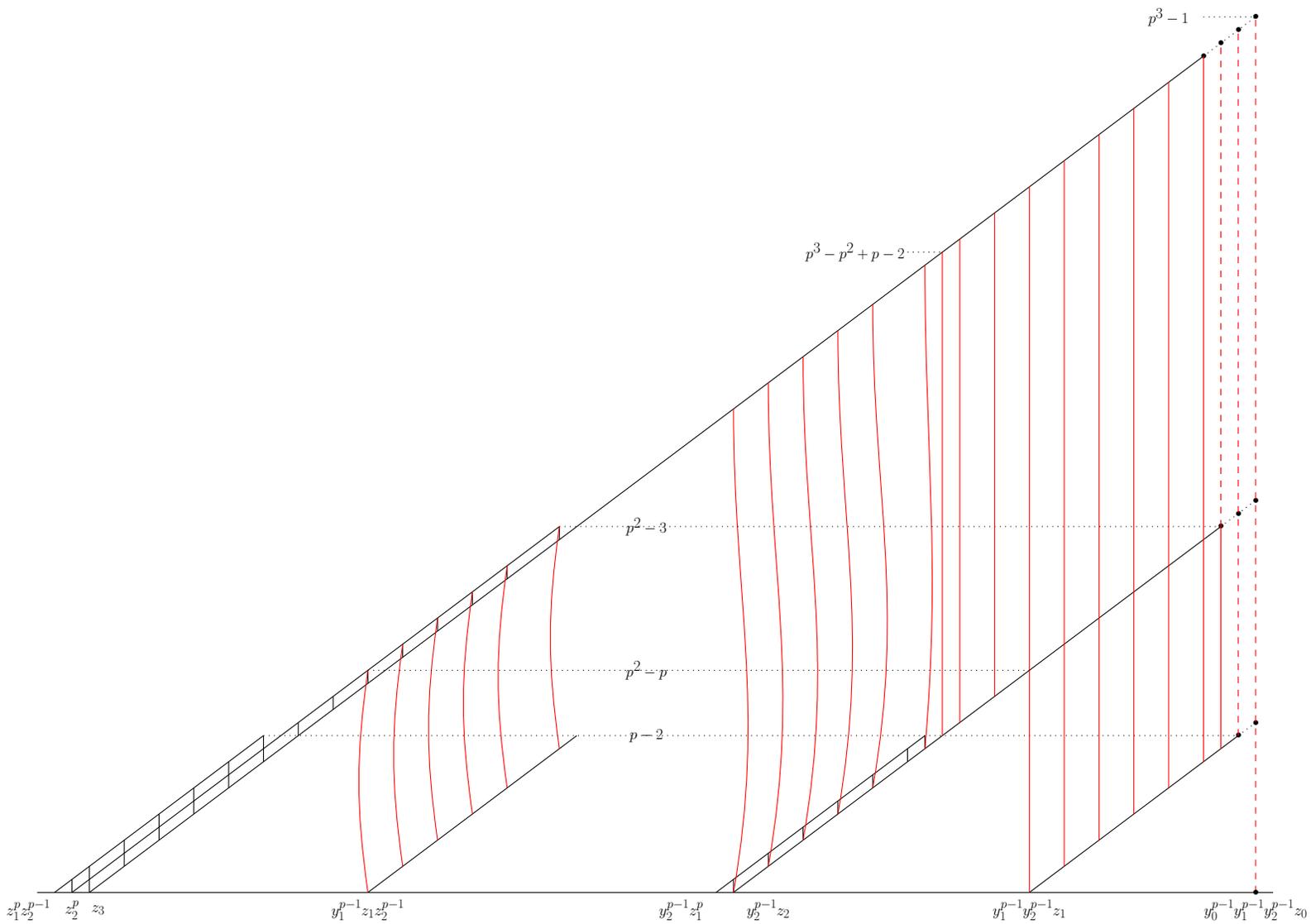
{\bf Schematic of $A_3$ and $B_3$ for odd $p$.}
\begin{center}
\begin{tikzpicture}
  \pic[rotate=90,scale=.28,transform shape] {testpic4};
\end{tikzpicture}
\end{center}
\end{fig}
\end{minipage}

\vfill\eject

Now we describe the portion of $ku^*(K_2)$ in odd gradings. Let $P[S]$ denote the polynomial algebra on a set $S$, and  $TP_i[S]=P[S]/(s^i:\ s\in S)$, the truncated  polynomial algebra. Let $\L_j=TP_p[z_i:\ i\ge j]$. Note that if $p=2$, $\L_j$ is an exterior algebra. For $i\le j$, let \begin{equation}\label{zij}z_{i,j}=z_i(z_i\cdots z_{j-1})^{p-1}.\end{equation} If $j=i$, then $z_{i,j}=z_i$.
\begin{defin}\label{Sdef} For $\ell>k\ge1$, let
$S_{k,\ell}=TP_{k+1}[v]\langle z_{k_0,\ell},\ldots,z_{\ell-k-1+k_0,\ell}\rangle$ with $pz_{i,\ell}=vz_{i-1,\ell}$ and $pz_{k_0,\ell}=0$.
\end{defin}

\ni For example, $S_{5,8}$ with $p=2$ is depicted in Figure \ref{S710}.

\bigskip
\begin{minipage}{6in}
\begin{fig}\label{S710}

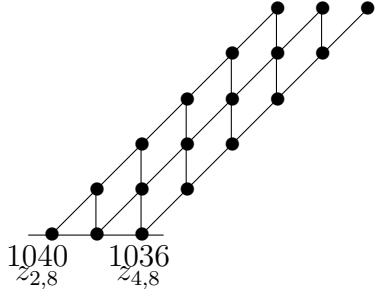
{\bf $S_{5,8}$ if $p=2$}

\begin{center}

\begin{\tz}[scale=.3]
\draw (0,0) -- (10,10);
\draw (2,0) -- (12,10);
\draw (4,0) -- (14,10);
\draw (2,0) -- (2,2);
\draw (4,0) -- (4,4);
\draw (6,2) -- (6,6);
\draw (8,4) -- (8,8);
\draw (10,6) -- (10,10);
\draw (12,8) -- (12,10);
\draw (-1,0) -- (5,0);

\node at (-.6,-.9) {$1040$};
\node at (3.9,-.9) {$1036$};
\node at (-.6,-1.9) {$z_{2,8}$};
\node at (3.9, -1.9) {$z_{4,8}$};

\node at (0,0) {\lb};
\node at (2,2) {\lb};
\node at (4,4) {\lb};
\node at (6,6) {\lb};
\node at (8,8) {\lb};
\node at (10,10) {\lb};
\node at (2,0) {\lb};
\node at (4,2) {\lb};
\node at (6,4) {\lb};
\node at (8,6) {\lb};
\node at (10,8) {\lb};
\node at (12,10) {\lb};
\node at (4,0) {\lb};
\node at (6,2) {\lb};
\node at (8,4) {\lb};
\node at (10,6) {\lb};
\node at (12,8) {\lb};
\node at (14,10) {\lb};

\end{\tz}
\end{center}
\end{fig}
\end{minipage}

The following result describes the portion of $ku^*(K_2)$ in odd gradings. The exponent of $p$ in an integer $i$ is denoted simply by $\nu(i)$; the prime $p$ is implicit. The element $q$ here has grading $9$ or $4p-1$, as mentioned earlier.
\begin{thm} \label{oddthm} There is an isomorphism of $ku^*$-modules
$$ku^{\text{odd}}(K_2)\approx \bigoplus_{i\ge1}\bigoplus_{\ell\ge\nu(i)+2}q y_1^{i-1}S_{\nu(i)+1,\ell}\ot TP_{p-1}[z_\ell]\ot\L_{\ell+1}.$$\end{thm}

The non-visual, formulaic form of our result is as follows.
\begin{thm}\label{formula} The $ku^*$-module $ku^*(K_2)$ is isomorphic to a trivial $ku^*$-module plus a module whose associated graded is
\begin{eqnarray}&&P[y_1]y_0^{p-1}z_0\oplus\bigoplus_{t\ge1}TP_{p^t}[v]\ot P[y_t]z_t\label{a1}\\
&\oplus&\label{a2}\bigoplus_{t\ge k_0}TP_{p^t-t}[v]\ot P[y_t]z_t\L_t\\
&\oplus&\label{a3}\bigoplus_{i\ge1}\bigoplus_{\ell\ge0}TP_{\nu(i)+2}[v]q y_1^{i-1}z_{k_0+\ell,\ell+\nu(i)+2}\overline\L_{\ell+\nu(i)+2}.\end{eqnarray}
Multiplication by $p$ in (\ref{a1}) and (\ref{a2}) is determined by (\ref{extns}) and in (\ref{a3}) as in Definition \ref{Sdef}.\end{thm}

Our initial interest in this project was $ku_*(K_2)$ (\cite{W},\cite{DW}), but  we first achieved success in computing $ku^*(K_2)$. In \cite[Example 3.4]{DD}, the following result was proved.
\begin{thm}\label{dual} There is an isomorphism of  $ku_*$-modules $ku_*(K_2)\approx(ku^{*+2p}K_2)^\vee$.\end{thm}
Here $M^\vee=\Hom(M,\Z/p^\infty)$, the Pontryagin dual, localized at $p$. A homotopy chart for $ku_*(K_2)$ could be thought of as a shifted version of
 the homotopy chart of $ku^*(K_2)$ viewed upside-down and backwards. For example, the element of $ku^{108}(K_2)^\vee$ dual to the element $v^4y_3z_3z_4$ in Figure \ref{B7} corresponds to the generator of a $\Z_4$ in $ku_{104}(K_2)$ on which $v^4$ acts nontrivially. This element can be seen in Figure \ref{H_*pic}.

  A remarkable property, for which one explanation is given in Section \ref{optsec}, is that  $B_k$ is self-dual as a $ku^*$-module. One way of stating this is to let $\widetilde B_k$ denote $B_k$ with its indices negated. Then there is an isomorphism of $ku_*$-modules
 \begin{equation}\Sigma^{2(p^{k+1}+p^{k}+(k+1)p-k+1)}\widetilde B_k\approx B_k^\vee.\label{Bd}\end{equation}
 For example, with $p=2$, the second smallest generator $Y$ of $\Sigma^{208}\widetilde B_5$ is in grading $208-134=74$ and has $2Y\ne0$ and $v^4Y\ne0$. (See Figure \ref{B7}.) The second generator $Z$ of $B_5^\vee$ is dual to the class in position $(74,4)$ in Figure \ref{B7}, and also satisfies $2Z\ne0$ and $v^4Z\ne0$.
 The isomorphism (\ref{Bd}) can be proved by induction on $k$ using Definition \ref{ABdef}.

 A complete description of the $ku_*$-module $ku_*(K_2)$ is immediate from Theorems \ref{evthm}, \ref{oddthm}, and \ref{dual}.  However, one might like a complete description of its ASS. We can write formulas for the $E_2$-term and differentials, but will not do so here. In Theorem \ref{ku*thm} we give a complete description of the $E_\infty$-term of the ASS of $ku_*(K_2)$ with exotic extensions included, in terms of the charts  described in Section \ref{intro}.

  In \cite{DD}, a comparison was made of a chart for $A_3$ and its $ku_*$ analogue. Here we
 present in Figure \ref{H_*pic} the $ku_*$ analogue of Figure \ref{B7}. This presents the portion of the ASS of $ku_*(K_2)$ dual to $A_5$ with $p=2$ under the isomorphism of Theorem \ref{dual}. The ASS chart dual to $B_5$ is obtained from this by removing the classes connected by dashed lines, and lowering the remaining tower so that the bottom is in filtration 0. The resulting chart is isomorphic to the $B_5$ part of Figure \ref{B7}.

 \vfill\eject

 \tikzset{
  testpic3/.pic=
{\draw (0,5) -- (1,6) -- (1,5) -- (5,9) -- (5,7) -- (29,31);
\draw (-3,0) -- (33,0);
\draw (3,7) -- (3,5);
\draw (4,6) -- (4,8);
\draw [dashed] (3,5) -- (-2,0) -- (-2,5);
\draw [dashed] (-2,1) -- (2,5);
\draw [dashed] (-2,1) -- (-1,2) -- (-1,1);
\draw [dashed] (0,2) -- (0,5) -- (-2,3);
\draw [dashed] (-2,2) -- (1,5) -- (1,3);
\draw [dashed] (-2,4) -- (-1,5) -- (-1,2);
\draw [dashed] (2,4) -- (2,5);
\draw [color=red] (5,0) to[out=98, in=262] (5,9);
\draw [color=red] (10,0) to[out=98, in=262] (10,13);
\draw [color=red] (11,1) to[out=98, in=262] (11,14);
\draw [color=red] (12,2) to[out=98, in=262] (12,15);
\draw [color=red] (13,3) to[out=98, in=262] (13,16);
\draw [color=red] (22,0) to[out=98, in=262] (22,4);
\draw [color=red] (22,3) to[out=83, in=270] (22,24);
\draw [color=red] (21,2) to[out=83, in=270] (21,23);
\draw [color=red] (20,1) to[out=83, in=270] (20,22);
\draw [color=red] (19,0) to[out=83, in=270] (19,21);
\draw [color=red] (27,0) to[out=83, in=270] (27,8);
\draw [color=red] (14,0) -- (14,4);
\draw [color=red] (23,4) -- (23,25);
\draw [color=red] (24,5) -- (24,26);
\draw [color=red] (25,6) -- (25,27);
\draw [color=red] (26,7) -- (26,28);
\draw [color=red] (27,8) -- (27,29);
\draw [color=red] (28,1) -- (28,30);
\draw [color=red] (29,2) -- (29,31);
\draw [color=red] (30,3) -- (30,11);
\draw [color=red] (31,0) -- (31,4);
\node at (0,5) {\lb};
\node at (1,5) {\lb};
\node at (1,6) {\lb};
\node at (2,6) {\lb};
\node at (3,7) {\lb};
\node at (4,8) {\lb};
\node at (5,9) {\lb};
\node at (2,5) {\lb};
\node at (3,6) {\lb};
\node at (4,7) {\lb};
\node at (5,8) {\lb};
\node at (6,9) {\lb};
\node at (7,10) {\lb};
\node at (8,11) {\lb};
\node at (9,12) {\lb};
\node at (10,13) {\lb};
\node at (11,14) {\lb};
\node at (12,15) {\lb};
\node at (13,16) {\lb};
\node at (3,5) {\lb};
\node at (4,6) {\lb};
\node at (5,7) {\lb};
\node at (6,8) {\lb};
\node at (7,9) {\lb};
\node at (8,10) {\lb};
\node at (9,11) {\lb};
\node at (10,12) {\lb};
\node at (11,13) {\lb};
\node at (12,14) {\lb};
\node at (13,15) {\lb};
\node at (14,16) {\lb};
\node at (15,17) {\lb};
\node at (16,18) {\lb};
\node at (17,19) {\lb};
\node at (18,20) {\lb};
\node at (19,21) {\lb};
\node at (20,22) {\lb};
\node at (21,23) {\lb};
\node at (22,24) {\lb};
\node at (23,25) {\lb};
\node at (24,26) {\lb};
\node at (25,27) {\lb};
\node at (26,28) {\lb};
\node at (27,29) {\lb};
\node at (28,30) {\lb};
\node at (29,31) {\lb};
\node at (5,0) {\lb};
\node at (6,1) {\lb};
\node at (9,0) {\lb};
\node at (10,1) {\lb};
\node at (10,0) {\lb};
\node at (11,1) {\lb};
\node at (12,2) {\lb};
\node at (13,3) {\lb};
\node at (14,4) {\lb};
\node at (14,0) {\lb};
\node at (15,1) {\lb};
\node at (17,0) {\lb};
\node at (18,1) {\lb};
\node at (18,0) {\lb};
\node at (19,1) {\lb};
\node at (20,2) {\lb};
\node at (21,3) {\lb};
\node at (22,4) {\lb};
\node at (19,0) {\lb};
\node at (20,1) {\lb};
\node at (21,2) {\lb};
\node at (22,3) {\lb};
\node at (23,4) {\lb};
\node at (24,5) {\lb};
\node at (25,6) {\lb};
\node at (26,7) {\lb};
\node at (27,8) {\lb};
\node at (28,9) {\lb};
\node at (29,10) {\lb};
\node at (30,11) {\lb};
\node at (22,0) {\lb};
\node at (23,1) {\lb};
\node at (26,0) {\lb};
\node at (27,1) {\lb};
\node at (27,0) {\lb};
\node at (28,1) {\lb};
\node at (29,2) {\lb};
\node at (30,3) {\lb};
\node at (31,4) {\lb};
\node at (31,0) {\lb};
\node at (32,1) {\lb};

\draw (17,0) -- (18,1) -- (18,0) -- (22,4);
\draw (19,0) -- (30,11);
\draw (19,0) -- (19,1) -- (20,2) -- (20,1) -- (21,2) -- (21,3) -- (22,4) -- (22,3);
\node at (-2,0) {\lb};
\node at (-2,1) {\lb};
\node at (-2,2) {\lb};

\node at (-2,3) {\lb};
\node at (-2,4) {\lb};
\node at (-2,5) {\lb};
\node at (-1,1) {\lb};
\node at (-1,2) {\lb};
\node at (-1,3) {\lb};
\node at (-1,4) {\lb};
\node at (-1,5) {\lb};
\node at (0,2) {\lb};
\node at (0,3) {\lb};
\node at (0,4) {\lb};
\node at (1,3) {\lb};
\node at (1,4) {\lb};
\node at (2,4) {\lb};

\draw (2,6) -- (2,5) -- (13,16) -- (13,15);
\draw (6,8) -- (6,9) -- (7,10) -- (7,9) -- (8,10) -- (8,11) -- (9,12) -- (9,11) -- (10,12) -- (10,13) -- (11,14) -- (11,13) -- (12,14) -- (12,15);
\draw (3,5) -- (5,7);
\draw (9,0) -- (10,1) -- (10,0) -- (14,4);
\draw (5,0) -- (6,1);
\draw (14,0) -- (15,1);
\draw (22,0) -- (23,1);
\draw (26,0) -- (27,1) -- (27,0) -- (31,4);
\draw (31,0) -- (32,1);
\node at (-2,-.4) {$64$};
\node at (5,-.4) {$78$};
\node at (10,-.4) {$88$};
\node at (14,-.4) {$96$};
\node at (17,-.4) {$102$};
\node at (22,-.4) {$112$};
\node at (26,-.4) {$120$};
\node at (32,-.4) {$132$};
}}

\bigskip
\begin{minipage}{6in}
\begin{fig}\label{H_*pic}

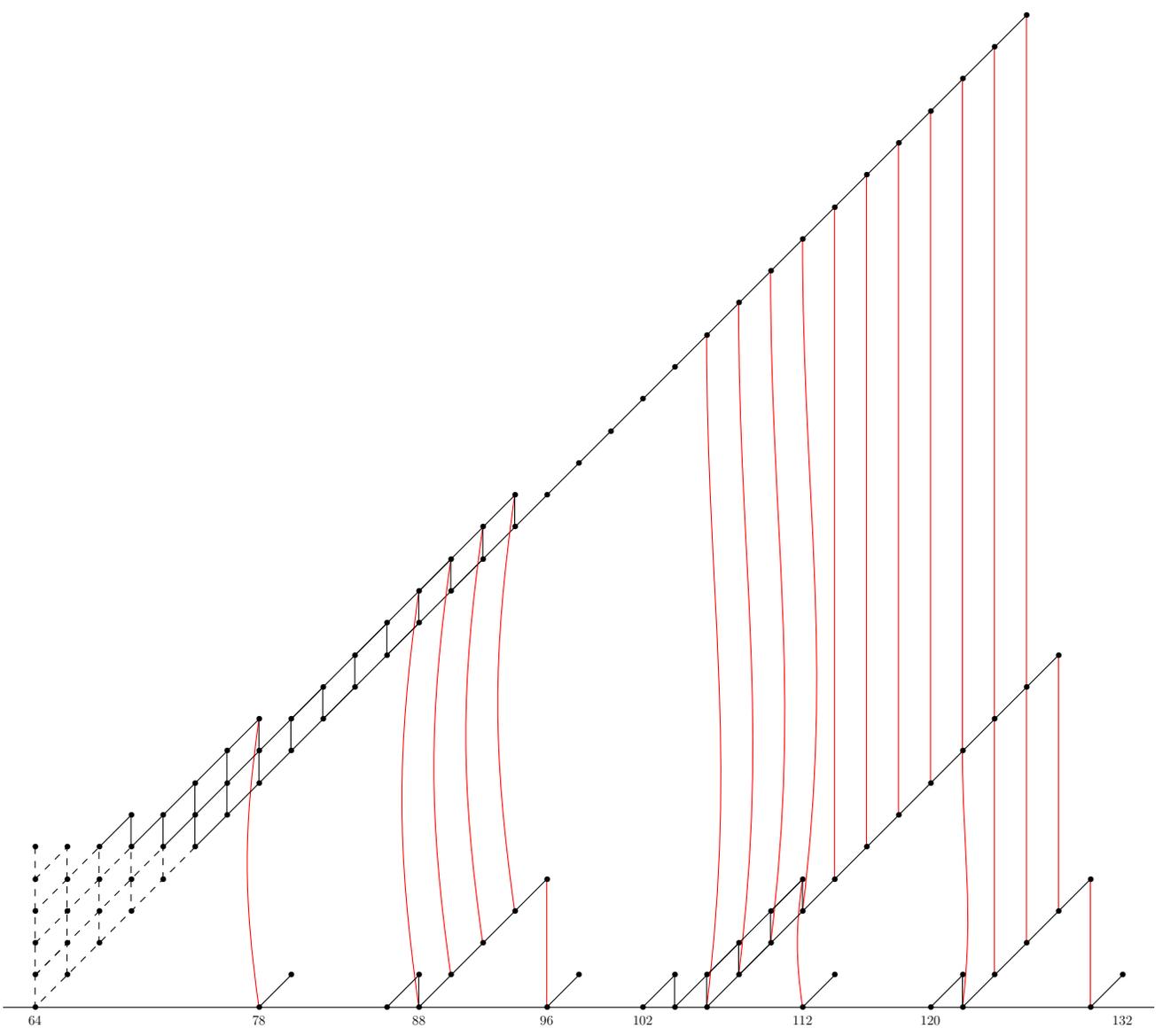
{\bf Portion of $ku_*(K_2)$ corresponding to $B_5$ and $A_5$.}
\begin{center}
\begin{tikzpicture}
  \pic[rotate=90,scale=.48,transform shape] {testpic3};
\end{tikzpicture}
\end{center}
\end{fig}
\end{minipage}

\bigskip
\vfill\eject
We observe that in even gradings of the ASS for $ku_*(K_2)$, $h_0$-extensions exactly correspond to exotic extensions in the ASS of $ku^{*+2p}(K_2)$, and vice versa.
 As a typical example of the duality, the summands of $ku^{82}(K_2)$, $ku^{82}(K_2)^\vee$, and $ku_{78}(K_2)$ in Figures \ref{B7} and \ref{H_*pic} are all isomorphic to $\Z_8\oplus\zt$. But for the $ku_*$-module structure, it is $ku^{82}(K_2)^\vee$ and $ku_{78}(K_2)$ that correspond, since in both, the element that is divisible by 4,  in position $(82,0)$ and $(78,7)$, resp., is also divisible by $v^7$ for $A_5$ and by $v^4$ for $B_5$.


\begin{thm}\label{ku*thm}  The $E_\infty$-term of the ASS of $ku_*(K_2)$ with exotic extensions included contains exactly the following.
 \begin{itemize}
 \item There is a trivial $ku_*$-module, which when $p=2$ has generators corresponding to those enumerated at the end of Section \ref{E2sec} with gradings decreased by 4, and similarly when $p$ is odd.
 \item For every $S_{k,\ell}$ occurring in a summand of Theorem \ref{oddthm}, there is a chart of the same form as Figure \ref{S710} with $v$-towers of height $k+1$ on generators  in gradings $2p^{\ell+1}+2(p-1)(i-k_0-1)$ for $1\le i\le \ell-k$. One must add to this the grading of the other factors accompanying $S_{k,\ell}$ in Theorem \ref{oddthm}.
\item For each occurrence of $B_k$ in Theorem \ref{evthm}, there is a summand  $$\Sigma^{2(p^{k+1}+p^{k}+kp-k+1)}\widetilde B_k$$ with gradings increased by those of other factors accompanying $B_k$ in \ref{evthm}. Here $\widetilde B_k$ is as defined prior to (\ref{Bd}).
\item For each summand $y_k^eA_k$ in Theorem \ref{evthm}, there is a variant of    $\Sigma^{2(p^{k+1}+p^{k}+kp-k+1)}\widetilde B_k$ with gradings increased by $2ep^k$. In this variant, the initial $v$-towers are pushed up by $k$ filtrations and surrounded with a triangle of classes of the sort appearing in the lower left corner of Figure \ref{H_*pic}. See Remark \ref{rk}.
\end{itemize}
\end{thm}

\begin{proof} Theorem \ref{dual} and our results for $ku^*(K_2)$ give    the $ku_*$-module structure of $ku_*(K_2)$, but that is not the same as the ASS picture. Expanding on work done in \cite{DW} and \cite{W} and using methods such as those in Section \ref{E2sec}, we were able to write the $E_2$-term of the ASS for $ku_*(K_2)$, and had conjectured the differentials (but not the extensions) prior to embarking on our $ku$-cohomology project. We were unable to {\em prove} the differentials, probably because we had not taken sufficient advantage of the exact sequence with $k(1)_*(K_2)$. Now that we know the 2-orders and $v$-heights of generators (by grading, at least, if not by name), it is straightforward to see that the differentials  must be as we expected. The isomorphism (\ref{Bd}) plays an important role here; the left hand side gives the ASS form of the right hand side.
\end{proof}

\begin{rmk}\label{rk}{\rm Regarding the unusual portion of the ASS chart for part of $ku_*(K_2)$ in the lower left of Figure \ref{H_*pic}, this is obtained from \cite[Fig.~4.2]{DW} with $d_6$-differentials on all odd-graded towers. For $A_k$, it will be a triangle going up to filtration $k$, with all but the first two dots on the top row being part of $B_k$.}\end{rmk}

The structure of the rest of the paper is as follows. In Section \ref{E2sec}, we compute the $E_2$-term of the ASS for $ku^*(K_2)$. In Section \ref{difflsec}
we determine the differentials in this ASS. In order to do so, we need to compare with $k(1)^*(K_2)$, where $k(1)$ is the spectrum for mod-$p$ connective $KU$-theory, using the exact sequence
\begin{equation}\label{LES}\to k(1)^{*-1}(K_2)\to ku^*(K_2) \mapright{p} ku^*(K_2)\to k(1)^*(K_2)\to ku^{*+1}(K_2)\mapright{p}.\end{equation}
In Section \ref{difflsec}, we restate results about $k(1)^*(K_2)$ from \cite{DRW}. At the end of Section 3, we  show how the descriptions of $ku^*(K_2)$ in Theorems \ref{evthm} and \ref{oddthm} are obtained once we know the differentials. This exact sequence is also used in
determining the exotic extensions of (\ref{extns}), which is done in Section \ref{extnsec}.  In Section \ref{LESsec}, we propose complete formulas for the exact sequence (\ref{LES}), and then in Section \ref{allsec}, we show that our proposed formulas account for all elements of $k(1)^*(K_2)$ exactly once.

The main point of Section \ref{allsec} is to prove that there are no additional exotic extensions in $ku^*(K_2)$. An exotic extension $p\cdot A=B$ implies that $A$ is  not in the image from $k(1)^{*-1}(K_2)$, and $B$ does not map nontrivially to $k(1)^*(K_2)$, so once we have shown that all elements are accounted for, there can be no more extensions.
Many of our formulas in Section \ref{LESsec} are forced by naturality. However, many others occur in regular families, but with surprising filtration jumps. We could probably prove that the homomorphisms {\em must} be as we claim, by showing that there are no other possibilities, but we prefer to forgo doing that. In the optional Section \ref{optsec}, we discuss in more detail how the charts are obtained and provide an explanation for the duality result (\ref{Bd}).

\section{The $E_2$-term of the ASS for $ku^*(K_2)$} \label{E2sec}

We will need some notation.  By $H^*K_2$, we understand $H^*(K(\zp,2);\zp)$.
Let $E$ denote an exterior algebra, $P$ a polynomial algebra, and
$TP_n[x]=P[x]/(x^n)$ the truncated polynomial algebra.  In all
cases these will be over $\zp$, the integers mod $p$. Let $\Ebar$ denote the augmentation ideal
of an exterior algebra, and
$E_1 = E[Q_0,Q_1]$, where $Q_i$ are the Milnor primitives.
Because $Q_i^2 = 0$ we have homology groups, $H_*(-;Q_i)$, defined for $E_1$-modules.
We let $\langle y_1, y_2, \ldots \rangle$ denote the $\zp$-span of classes
$y_i$.

The Adams spectral sequence (ASS) for $ku^*(K_2)$ has $E_2^{s,t} =
\Ext_{\mathcal{A}}^{s,t}(H^*(bu),H^*K_2)$, where $\mathcal{A}$ is the mod $p$
Steenrod algebra and $H^*(bu) \iso \mathcal{A}/\mathcal{A}(Q_0,Q_1)$.
Using a standard change of rings theorem, \cite{Liu64}, this is
$\Ext_{E_1}^{s,t}(\zp,H^*K_2)$.  This converges to $ku^{-(t-s)}(K_2)$.
We depict this with $E_2^{s,t}$ in position $(t-s,s)$ as usual, but
label the axis with codegrees, the negative of the homotopical degree, so the
left side of the chart will have positive gradings and refer to cohomological grading.
In an attempt to avoid confusion, we rewrite this as $G_2^{-(t-s),s}$.  With
this notation, the differentials are $d_r : G_r^{a,b} \lra G_r^{a+1,b+r}$,
multiplication by the element $v \in ku^{-2(p-1)}$ (also considered in $G_r^{-2(p-1),1}$),
is $v : G_r^{a,b} \lra G_r^{a-2(p-1),b+1}$, and multiplication by the element representing
$p \in ku^0$,  ($h_0 \in G_r^{0,1}$), is $h_0 : G_r^{a,b} \lra G_r^{a,b+1}$.

In the paragraph preceding Remark \ref{2.18}, we will define elements $z_j \in G_2^{2(p^{j+1}+1),0}$ for $j \ge 0$ and elements
$$z_{i,j} \in G_2^{2(p^{j+1}+1+(p-1)(j-i)),0}$$ as in (\ref{zij}) satisfying the properties in Definition \ref{Sdef}.
\begin{defin}\label{Wdef}
For $j \ge k_0$, we define $W_j =
\langle z_{j,j},z_{j-1,j},\ldots, z_{k_0,j}\rangle$.\end{defin}
We also have
$y_i \in G_2^{2p^i,0}$ for $i\ge0$, and \begin{equation}\label{qdef}q\in G^{9,0}_2\text{ if }p=2,\text{ and in }G^{4p-1,0}_2\text{ if $p$ is odd.}\end{equation}
One last definition, let
$\L_{j+1}=TP_p[z_i:\,i\ge j+1]$.

A picture of $P[v]\otimes W_5$ as a $P[v,h_0]$-module with $p=2$ appears in Figure \ref{fig2}.
\bigskip

\begin{minipage}{5in}

\begin{fig}\label{fig2}

{\bf A depiction of $P[v]\otimes W_5$}

\begin{center}
\begin{\tz}[scale=.55]
\draw (-1,0) -- (10,0);
\draw [->] (0,0) -- (11,5.5);
\draw [->]  (2,0) -- (11,4.5);
\draw [->]  (4,0) -- (11,3.5);
\draw [->] (6,0) -- (11,2.5);
\draw (2,0) -- (2,1);
\draw (4,0) -- (4,2);
\draw (6,0) -- (6,3);
\draw (8,1) -- (8,4);
\draw (10,2) -- (10,5);
\node at (0,-.5) {$136$};
\node at (2,-.5) {$134$};
\node at (4,-.5) {$132$};
\node at (6,-.5) {$130$};
\node at (0,-1.3) {$z_{2,5}$};
\node at (2,-1.3) {$z_{3,5}$};
\node at (4,-1.3) {$z_{4,5}$};
\node at (6,-1.3) {$z_{5,5}$};

\node at (0,0) {$\ssize\bullet$};
\node at (2,1) {$\ssize\bullet$};
\node at (4,2) {$\ssize\bullet$};
\node at (6,3) {$\ssize\bullet$};
\node at (8,4) {$\ssize\bullet$};
\node at (10,5) {$\ssize\bullet$};
\node at (2,0) {$\ssize\bullet$};
\node at (4,1) {$\ssize\bullet$};
\node at (6,2) {$\ssize\bullet$};
\node at (8,3) {$\ssize\bullet$};
\node at (10,4) {$\ssize\bullet$};
\node at (4,0) {$\ssize\bullet$};
\node at (6,1) {$\ssize\bullet$};
\node at (8,2) {$\ssize\bullet$};
\node at (10,3) {$\ssize\bullet$};
\node at (6,0) {$\ssize\bullet$};
\node at (8,1) {$\ssize\bullet$};
\node at (10,2) {$\ssize\bullet$};
\end{\tz}
\end{center}
\end{fig}
\end{minipage}

\bigskip
The remainder of this section is devoted to the proof of the following result.
\begin{thm}\label{E2}
The $E_2$ term of the Adams spectral sequence for the reduced $ku^*(K_2)$ is
isomorphic as a $P[h_0,v]$-module to
$$
P[v,y_1]\ot E[q]
\ot\bigl( \bigoplus_{j\ge k_0}(W_j\ot TP_{p-1}[z_j]\ot\L_{j+1})\bigr)
$$
$$
\oplus
\bigl(P[h_0,v,y_1]\ot
E[ v^{k_0}q]\bigr)
\oplus
\biggl(
P[y_1]\ot\begin{cases}\langle y_0^{p-1}z_0\rangle&p\text{ odd}\\
\langle y_0z_0,z_1,h_0y_0z_0=vz_1\rangle&p=2.\end{cases}
\biggr)
$$
\noindent
plus a trivial $P[h_0,v]$-module.
\end{thm}

Some of the algebra structure of this $E_2$ will be useful later. For example, the product structure among the $z_j$'s will be clear, and also the formula
\begin{equation}\label{x9z4}(v^2q)^2 = v^4z_2,\end{equation}
holds when $p=2$ since, as we shall see, in $H^*(K_2)$, $x_9^2-Q_0x_{17}\in\im(Q_1)$.

We will give a detailed proof when $p=2$, and then sketch the minor changes for odd $p$.
There are two parts to proving this theorem.  First, we must give
a complete description of the $E_1$-module structure of $H^*K_2$.
Second, we have to compute $\Ext_{E_1}^{*,*}(\zt,-)$ of this.
We begin the first part.

Serre (\cite{Ser}) showed that $H^*K_2$ is a
polynomial algebra on classes $u_{2^j+1}$ in degree $2^j+1$ for $j\ge0$
defined by $u_2=\io_2$ and $u_{2^{j+1}+1}=\sq^{2^j}u_{2^j+1}$ for $j\ge0$. We
easily have
$$
Q_0(u_2)=u_3,\ Q_0(u_3)=0,\ Q_0(u_{2^j+1})=u_{2^{j-1}+1}^2\text{ for }j\ge2,
$$
and
$$
Q_1(u_2)=u_5,\ Q_1(u_3)=u_3^2,\ Q_1(u_5)=0,\ Q_1(u_{2^j+1})=u_{2^{j-2}+1}^4\text{ for }j\ge3.
$$
Let $x_5=u_5+u_2u_3$ and write $H^*K_2$ as an associated graded object:
$$
P[u_2^2]\ot E[x_5] \ot \bigl( E[u_2] \ot P[u_3] \bigr) \ot_{j\ge 2}
\left( E[u_{2^{j+1}+1}]\ot P[(u_{2^j+1})^2] \right)
$$
From this, we can read off
\begin{lem}
\label{Q0}
$$
H_*(H^*K_2;Q_0)=P[u_2^2]\ot E[x_5]
$$
\end{lem}
Letting   $x_9=u_9+u_3^3$ and $x_{17}=u_{17}+u_2u_5^3$,
we rewrite again as

\begin{gather*}
P[u_2^2] \ot TP_4[x_9]\ot TP_4[x_{17}]\ot_{j>4} E[(u_{2^j+1})^2]\\
\ot \bigl( E[u_2]\ot P[u_5]\bigr)
\ot \bigl( E[u_3] \ot P[u_3^2]\bigr) \ot_{j>4}
\bigl(E[u_{2^j+1}] \ot P[(u_{2^{j-2}+1})^4]\bigr).
\end{gather*}

Again we read off
\begin{lem}\label{Q1}
$$
H_*(H^*K_2;Q_1) =
P[u_2^2] \otimes TP_4[x_9] \otimes TP_4[x_{17} ]
\otimes_{j>4} E[(u_{2^j+1})^2]
$$
\end{lem}
An associated graded version of this is
\begin{lem}\label{Q1g}
$$
H_*(H^*K_2;Q_1) =
P[u_2^2] \otimes E[x_9] \otimes E[x_{17} ]
\otimes_{j>2} E[(u_{2^j+1})^2]
$$
\end{lem}

\noindent The bulk of the work here is finding a nice splitting of $H^*K_2$ as an $E_1$-module.

Let $N$ be the $E_1$-submodule with single nonzero elements in
gradings 5, 7, 8, 9, and 10 with generators $x_5=u_5 + u_2 u_3$, $x_7=u_2u_5$,
and $x_9= u_9 + u_3^3$, satisfying $Q_0x_7=Q_1x_5$ and $Q_0x_9=Q_1x_7=x_{10}$.
It has a $Q_0$-homology class $x_5$ and a $Q_1$-homology class $x_9$. This class $x_9$ is called $q$ in Theorem \ref{E2} and in all other sections.
A picture of $N$ is in Figure \ref{N}. In pictures such as this, straight lines indicate $Q_0=\sq^1$ and curved lines $Q_1$.

\bigskip
\begin{minipage}{6in}
\begin{fig}\label{N}

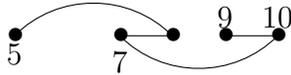
{\bf An $E_1$-module $N$.}

\begin{center}

\begin{\tz}[scale=.35]

\draw (4,0) -- (6,0);
\draw (8,0) -- (10,0);
\draw (0,0) to[out=45, in=135] (6,0);
\draw (4,0) to[out=315, in=225] (10,0);
\node at (0,-.7) {$5$};
\node at (8,.7) {$9$};
\node at (4,-1) {$7$};

\node at (10,.7) {$10$};
\node at (0,0) {\lb};

\node at (4,0) {\lb};
\node at (6,0) {\lb};
\node at (8,0) {\lb};
\node at (10,0) {\lb};
\end{\tz}
\end{center}
\end{fig}
\end{minipage}

\bigskip

The $E_1$-submodule $P[u_2^2]\oplus P[u_2^2]\ot N$
carries the $Q_0$-homology of $H^*K_2$, while the remaining $Q_1$-homology is,
written in our usual way as an associated graded version,
\begin{equation}\label{rest}
P[u_2^2]\ot E[x_9]\ot\Ebar[x_{17},u_{2^j+1}^2,\ j > 2].
\end{equation}

We will exhibit a $Q_0$-free $E_1$-submodule $R$ whose $Q_1$-homology
is exactly the above $\Ebar$. Moreover, $N\ot R$ contains an $E_1$-split
summand $S$ which maps isomorphically to $\langle x_9\rangle\ot R$.

It is premature to state this because we haven't defined $R$ and $S$ yet,
but for the record:

\begin{prop}
\label{T}
As an $E_1$ module, $\Ht^*K_2$ is isomorphic to $T \oplus F$ where $F$ is
 free over $E_1$ and $T$ is
$$
P[u_2^2]\ot
\bigl(\langle u_2^2\rangle\oplus N\oplus R\oplus S\bigr)
$$
\end{prop}

\begin{center}
\textbf{A start on $R$ and $S$.}
\end{center}

For this to make sense, we need to find $R$ and $S$.
The module $R$ is a direct sum of shifted versions of modules $L_k$,
$k \ge 0$,
which have generators $g_{2i}$, $0\le i\le k$, with $Q_1g_{2i}=Q_0g_{2i+2}$
for $0\le i<k$, $Q_0g_0\ne0$, and $Q_1g_{2k}=0$.
For example, $L_3$ is depicted in Figure \ref{L3}.

\bigskip
\begin{minipage}{6in}
\begin{fig}\label{L3}

{\bf The $E_1$-module $L_3$.}

\begin{center}

\begin{\tz}[scale=.35]
\draw (0,0) -- (2,0);
\draw (4,0) -- (6,0);
\draw (8,0) -- (10,0);
\draw (12,0) -- (14,0);
\draw (0,0) to[out=45, in=135] (6,0);
\draw (4,0) to[out=315, in=225] (10,0);
\draw (8,0) to[out=45, in=135] (14,0);
\node at (0,-.7) {$g_0$};
\node at (4,-1) {$g_2$};
\node at (8,1) {$g_4$};
\node at (12,-.7) {$g_6$};
\node at (0,0) {\lb};
\node at (2,0) {\lb};
\node at (4,0) {\lb};
\node at (6,0) {\lb};
\node at (8,0) {\lb};
\node at (10,0) {\lb};
\node at (12,0) {\lb};
\node at (14,0) {\lb};
\end{\tz}
\end{center}
\end{fig}
\end{minipage}

\bigskip

A splitting
map, $\langle x_9\rangle \ot L_k \lra N \ot L_k$, for the
epimorphism $N\ot L_k\to \langle x_9\rangle \ot L_k$  is defined by
 $$ x_9 g_{2i} \mapsto x_9\ot g_{2i}+x_7\ot g_{2i+2}+x_5\ot
g_{2i+4}\text{ for }0\le i\le k-2,$$ $x_9 g_{2k-2} \mapsto x_9\ot g_{2k-2}+x_7\ot g_{2k}$,
and $x_9 \ot g_{2k} \mapsto x_9\ot g_{2k}$.

\bigskip

\begin{center}
\textbf{The $E_1$-module $M_j$}
\end{center}

Let
$$
x_{2^j+1}=u_{2^j+1}+
\begin{cases}u_2u_5^3&j=4\\
u_2u_3u_5^2u_9^2&j=5\\
u_3u_5^2u_9^2u_{17}^2&j=6\\
0&j>6
\end{cases}
\text{ and }
w_{2^j-1}=
\begin{cases}u_2u_3u_5^2&j=4\\
u_3u_5^2u_9^2&j=5\\
0&j>5.
\end{cases}$$
Then $Q_0x_{2^j+1}=u_{2^{j-1}+1}^2+Q_1w_{2^j-1}$, so $Q_0x_{2^j+1}$
and $u_{2^{j-1}+1}^2$ represent the same
$Q_1$-homology class.
Define $E_1$-modules $M_j$
inductively by $M_3=0$, and for $j\ge4$ there is a short exact sequence of $E_1$-modules
\begin{equation}\label{Mdef}
0\to u_{2^{j-2}+1}^2M_{j-1}\to M_j\to M_j'\to0,
\end{equation}
where $M_j'=\langle x_{2^j+1},Q_0x_{2^j+1}\rangle$
and $Q_1x_{2^j+1}=u_{2^{j-2}+1}^2Q_0x_{2^{j-1}+1}$.
The above definitions of the $x_{2^j+1}$
are necessary to get this formula to work right.

There is an isomorphism of $E_1$-modules
$M_j\approx\Sigma^{2^j+1}L_{j-4}$  given by

\begin{equation}\label{LM}
\Sigma^{2^j+1} g_{2i} \mapsto
\begin{cases}
x_{2^j+1} & i = 0 \\
u_{2^{j-2}+1}^2 x_{2^{j-1}+1} & i = 1 \\
u_{2^{j-2}+1}^2
u_{2^{j-3}+1}^2
x_{2^{j-2}+1} & i = 2 \\
u_{2^{j-2}+1}^2
u_{2^{j-3}+1}^2
\cdots
u_{2^{j-i-1}+1}^2
x_{2^{j-i}+1} & 2 < i \le j-4 \\
\end{cases}
\end{equation}

And we have

\begin{equation}\label{two}
H_*(M_j;Q_1)=
\begin{cases}
\langle u_9^2 , u_{17} \rangle & j = 4 \\
\langle u_{17}^2 , u_9^2 u_{17} \rangle & j = 5 \\
\langle u_{33}^2 , u_{17}^2 u_9^2 u_{17} \rangle & j = 6 \\
\langle u_{2^{j-1}+1}^2,u_{2^{j-2}+1}^2\cdots u_9^2x_{17}\rangle & j > 6 \\
\end{cases}
\end{equation}

\bigskip

\begin{center}
\textbf{The $E_1$-module $R$}
\end{center}

Let \begin{equation}\label{R}R=\bigoplus_{j\ge4}M_j\ot
E[u_{2^j+1}^2,u_{2^{j+1}+1}^2,\ldots].\end{equation}
Then $H_*(R;Q_1)=\Ebar[x_{17},u_9^2,u_{17}^2,\ldots]$,
since monomials in $\Ebar$ without $x_{17}$ appear from a first
term (of the two in (\ref{two})) in $H_*(M_j\ot E;Q_1)$, where $j$ is
minimal such that $u_{2^{j-1}+1}^2$ appears in the monomial,
while those with $x_{17}$, and also containing a
product $u_9^2\cdots u_{2^{j-2}+1}^2$ of maximal length, occur as a
second term in $H_*(M_j\ot E;Q_1)$.

\begin{proof}[Proof of Proposition \ref{T}]
We have the $E_1$-submodule $T$ given in Proposition \ref{T}.  Because this contains all
of the $Q_0$ and $Q_1$ homology, what remains must be free over $E_1$
by \cite{Wal62}.
\end{proof}

\begin{proof}[Proof of Theorem \ref{E2}]
We compute $\ext_{E_1}(\zt,T)$
with $T$ as in Proposition \ref{T}. We will not be concerned with the
free $E_1$-module $F$
but later we will give the
Poincar\'e
series for it.
Each copy of $E_1$ in $F$ gives a $\zt$ in $G^{*,0}$ that corresponds to
$Q_0Q_1$.


That
$$
\ext_{E_1}^{*,*}(\zt,P[u_2^2])=P[v,h_0,y_1]
$$
with $y_1\in G_2^{4,0}$ should be clear, given our labeling conventions.
We normally work with the reduced cohomologies, so the $y_1^0$ generator above
would be ignored. The $y_1$ notation is particularly useful when we consider all primes $p$. It is $y_0^{p^1}$ where $y_0\in G_2^{2,0}$. So $|y_1|=2p$.

We compute $\ext_{E_1}(\zt,N)$ in two ways using two different filtrations
of $N$.
From this we
see that the
generator of the towers can be thought of either as $v^2x_9$ or $h_0^2x_5$.

Using Figure \ref{N} as our guide,
our first filtration is
$\langle x_5, x_{8} \rangle $,
$\langle x_7, x_{10} \rangle $,
and
$\langle x_9 \rangle $.
The $\Ext$ on $x_9\in G^{9,0}$ is just $P[v,h_0]$.  For the other two, we get $h_0$-towers
on $x_{10}\in G^{10,0}$ and $x_8\in G^{8,0}$.
The extensions in $N$ show these two $h_0$-towers are connected by multiplication by $v$.
In addition,
a $d_1$ is forced on us by the extensions.
Figure \ref{fig3} describes this completely.

\medskip

\begin{minipage}{6in}
\begin{fig}\label{fig3}

{\bf The first computation of $\ext_{E_1}(\zt,N)$}

\begin{center}

\begin{\tz}[scale=.45]
\draw (-1,0) -- (7,0);
\draw (0,0) -- (0,5);
\draw [dotted] (0,0) -- (2,1);
\draw [dotted] (0,1) -- (2,2);
\draw [dotted] (0,2) -- (2,3);
\draw [dotted] (1,0) -- (7,3);
\draw [dotted] (1,1) -- (7,4);
\draw [dotted] (1,2) -- (7,5);
\draw (1,0) -- (1,5);
\draw (2,0) -- (2,5.5);
\draw (3,1) -- (3,5.5);
\draw (5,2) -- (5,6);
\draw (7,3) -- (7,7);
\draw [- >] (1,0) -- (.1,.9);
\draw [- > ] (1,1) -- (.1,1.9);
\draw [->] (1,2) -- (.1,2.9);
\draw [->] (3,1) -- (2.1,1.9);
\draw [->] (3,2) -- (2.1,2.9);
\draw [->] (3,3) -- (2.1,3.9);
\node at (0,-.5) {$10$};
\node at (2,-.5) {$8$};
\node at (5,-.5) {$5$};
\node at (7,-.5) {$3$};
\draw (9.5,0) -- (17.5,0);
\draw (12,0) -- (12,1);
\draw [dotted] (10,0) -- (12,1);
\draw (15,2) -- (15,6);
\draw [dotted] (15,2) -- (17,3);
\draw [dotted] (15,3) -- (17,4);
\draw [dotted] (15,4) -- (17,5);
\draw (17,3) -- (17,7);
\node at (10,-.5) {$10$};
\node at (12,-.5) {$8$};
\node at (15,-.5) {$5$};
\node at (17,-.5) {$3$};
\node at (8.5,2) {$\Rightarrow$};
\node at (14,2) {$v^2x_9$};
\end{\tz}
\end{center}
\end{fig}
\end{minipage}

\medskip

Again referring to Figure \ref{N}, our second filtration is
$\langle x_9, x_{10} \rangle$,
$\langle x_7, x_{8} \rangle$,
and
$\langle x_5 \rangle$.
Now our $\Ext$ groups are $P[v,h_0]$ on $x_5 \in G^{5,0}$,
$P[v]$ on $x_8 \in G^{8,0}$ and $x_{10} \in G^{10,0}$.
Again, the $d_1$ is forced by the extensions in $N$.
Figure \ref{fig4} describes the result.

\medskip

\begin{minipage}{6in}
\begin{fig}\label{fig4}

{\bf The second computation of $\ext_{E_1}(\zt,N)$}

\begin{center}

\begin{\tz}[scale=.45]
\draw (-1,0) -- (7,0);
\draw [dotted] (0,0) -- (8,4);
\draw [dotted] (2,0) -- (8,3);
\draw (2,0) -- (2,1);
\draw (4,1) -- (4,2);
\draw (6,2) -- (6,3);
\draw (5,0) -- (5,6);
\draw [dotted] (5,0) -- (7,1);
\draw [dotted] (5,1) -- (7,2);
\draw [dotted] (5,2) -- (7,3);
\draw [dotted] (5,3) -- (7,4);
\draw [dotted] (5,4) -- (7,5);
\draw (7,1) -- (7,7);
\draw [->] (5,0) -- (4.1,.9);
\draw [->] (5,1) -- (4.1,1.9);
\draw [->] (7,1) -- (6.1,1.9);
\draw [->] (7,2) -- (6.1,2.9);
\draw (9.5,0) -- (17.5,0);
\draw (12,0) -- (12,1);
\draw  [dotted] (10,0) -- (12,1);
\draw (15,2) -- (15,6);
\draw [dotted] (15,2) -- (17,3);
\draw [dotted] (15,3) -- (17,4);
\draw [dotted] (15,4) -- (17,5);
\node at (8.5,2) {$\Rightarrow$};
\draw (17,3) -- (17,7);
\node at (10,-.5) {$10$};
\node at (12,-.5) {$8$};
\node at (15,-.5) {$5$};
\node at (17,-.5) {$3$};
\node at (8.5,2) {$=$};
\node at (14,2) {$h_0^2x_5$};
\node at (0,-.5) {$10$};
\node at (2,-.5) {$8$};
\node at (5,-.5) {$5$};
\node at (7,-.5) {$3$};
\end{\tz}
\end{center}
\end{fig}
\end{minipage}

\medskip

This concludes the computation
of $\Ext$ for
$P[u_2^2]\otimes (\langle u_2^2 \rangle \oplus N)$ of
Proposition \ref{T}.  The result is the second line of Theorem \ref{E2}.

We need to compute $\Ext$ for $P[u_2^2] \otimes (R \oplus S)$
and show it is the same as the top line in Theorem \ref{E2}.
Since $S \iso \langle x_9 \rangle \otimes R$, all we need to
do is $P[u_2^2] \otimes R$ and ignore the $E[x_9]$.
Similarly we can ignore the $P[u_2^2]$ and the $P[y_1]$ because for
every power of $u_2^2$ we will have a copy of the answer indexed
by powers of $y_1$.  All we have left now is $R$, but $R$ is
just many copies of the various $M_j$ and the indexing for the
number of copies is given by the $\L_{j+1}$.

All that remains
is to show that $\Ext_{E_1}(\zt,M_j) \iso P[v]\otimes W_{j-2}$ with $W_{j-2}$ as in Definition \ref{Wdef}.\footnote{The reason for this awkward shift is that the gradings for $z_j$ which give the elegant statements in Definition \ref{ABdef} and elsewhere are not particularly convenient in developing the $E_2$ statement.}
Recall that $M_j = \Sigma^{2^j+1} L_{j-4}$.  We can filter $L_{j-4}$
into pairs of elements $g_{2i}, Q_0 g_{2i}$, for $0 \le i \le j-4$.
Then $\ext_{E_1}(\zt,M_j)$ has a $P[v]$ on each element $\Sigma^{2^j+1}Q_0 g_{2i}$
which we denote by $z_{j-i-2,j-2} \in G^{2^j+2+2i,0}$. The element $z_{j-2,j-2}$ is often called $z_{j-2}$. There is no $d_1$, but
undoing the filtration does solve the extension problem and gives
us $h_0 z_{k,j-2} = v z_{k-1,j-2}$.  This completes our computation
and thus our proof.
\end{proof}

\begin{rmk}\label{2.18}{\rm
To illustrate
the last computation in the proof,
consider
the generators of the $v$-towers for $\ext_{E_1}(\zt,M_7)$.
They
are $z_5$, $z_4^2$, $z_3^2z_4$, and $z_2^2z_3z_4$, which is what we have
called $z_{5,5}$, $z_{4,5}$, $z_{3,5}$, and $z_{2,5}$, as pictured in
Figure \ref{fig2}. For future reference, we note that (with $\sim$ meaning homologous)}
\begin{equation}\label{zj}z_j=Q_0x_{2^{j+2}+1}\sim u_{2^{j+1}+1}^2=Q_0u_{2^{j+2}+1}=Q_0Q_{j+2}\io_2=Q_{j+2}Q_0\io_2.\end{equation}
\end{rmk}

We now describe briefly the changes required when $p$ is odd. We have
$$H^*(K_2)=P[y_0]\ot P[g_1,g_2,\ldots]\ot E[u_0,u_1,\ldots],$$
with $|y_0|=2$, $|g_j|=2(p^j+1)$, $|u_i|=2p^i+1$, $Q_0y_0=u_0$, $Q_0u_i=g_i$, $Q_1y_0=u_1$, $Q_1u_0=g_1$, $Q_1u_i=g_{i-1}^p$, $i\ge2$. Let $y_1=y_0^p$. Then, similarly to the case $p=2$,
$$H_*(H^*K_2,Q_0)=P[y_1]\ot E[y_0^{p-1}u_0].$$
Let $N=\langle y_0^{p-1}u_0, q=y_0^{p-1}u_1, Q_0q=Q_1(y_0^{p-1}u_0)\rangle$. Then $P[y_1]\oplus P[y_1]\ot N$ carries the $Q_0$-homology and part of the $Q_1$-homology. Similarly to (\ref{rest}), the rest of the $Q_1$-homology is
$$P[y_1]\ot E[q]\ot \overline{E[w_1]\ot TP_p[g_2,g_3,\ldots]},$$
where $w_1=u_2+u_0g_1^{p-1}$. There are $E_1$-submodules $M_j$ for $j\ge2$, defined inductively by $M_2=\langle w_1,g_2=Q_0w_1\rangle$, $M_j'=\langle u_j,g_j=Q_0u_j \rangle$ for $j\ge3$, and for $j\ge3$, there exists a short exact sequence of $E_1$-modules
$$0\to g_{j-1}^{p-1}M_{j-1}\to M_j\to M_j'\to0,$$ with $Q_1u_j=g_{j-1}^p$. There is an isomorphism of $E_1$-modules $M_j\approx \Sigma^{2p^j+1}L_{j-2}$, where $L_j$ is similar to Figure \ref{L3}, but with $i$th generator ($i\ge0$) in grading $2(p-1)i$ rather than $2i$.

Let
$$R=\bigoplus_{j\ge2}M_j\ot TP_{p-1}[g_j]\ot TP_p[g_{j+1},\ldots].$$
Then
$H_*(R;Q_1)=\overline{E[w_1]\ot TP_p[g_2,g_3,\ldots]}$, and so, similarly to Proposition \ref{T}, up to free $E_1$-modules
\begin{equation}\label{Rp}H^*K_2\approx P[y_1]\ot(\langle y_1\rangle\oplus N\oplus R\oplus qR).\end{equation}
Similarly to Figure \ref{fig4}, $\ext_{E_1}(\zp,N)$ can be read off from Figure \ref{Nfig}. This gives the third summand and $vq$ part of the second summand in Theorem \ref{E2}, while the $\langle y_1\rangle$ part of (\ref{Rp}) gives the non-$vq$ part of the second summand. For the first summand in Theorem \ref{E2}, we replace $g_j$ by $z_{j-1}$, and then note that $\ext_{E_1}(\zp,M_j)\approx P[v]\ot W_{j-1}$, similar to Figure \ref{fig2}. For example, $M_3$ has $v$-towers on $g_3$ and $g_2^p$, which are renamed $z_2=z_{2,2}$ and $z_1^p=z_{1,2}$, the generators of the $v$-towers of $W_2$. This completes our sketch of proof of Theorem \ref{E2} when $p$ is odd.

\bigskip
\begin{minipage}{6in}
\begin{fig}\label{Nfig}

{\bf Computation of $\ext_{E_1}(\zp,N)$}

\begin{center}

\begin{\tz}[scale=.45]
\draw (-1,0) -- (25,0);
\draw (9,0) -- (9,6);
\draw (19,1) -- (19,7);
\node at (0,0) {\lb};
\node at (8,1) {\lb};
\node at (18,2) {\lb};
\node at (9,1) {\lb};
\node at (9,2) {\lb};
\node at (19,2) {\lb};
\node at (19,3) {\lb};
\node at (0,-.8) {$4p$};
\node at (0,-1.8) {$y_0^{p-1}g_1$};
\node at (13,-1) {$vq$};
\draw [->] [dashed] (13,-.7) -- (9.2,.8);
\node at (9,-.8) {$2p+1$};
\node at (19,-.8) {$3$};
\draw [->] (9,0) -- (8.3,.7);
\draw [->] (19,1) -- (18.3,1.7);
\node at (22,3) {$\cdot$};
\node at (22.5,3.5) {$\cdot$};
\node at (23,4) {$\cdot$};
\end{\tz}
\end{center} \end{fig}
\end{minipage}

\bigskip
We explain here the reason for the $k_0$ in Definition \ref{ABdef}. In Theorem \ref{E2},  $y_0^{p-1}z_0$ and $z_1$ are in the part that is not multiplied by higher $z$'s when $p=2$, but when $p$ is odd, they form the module $M_2$, whose Ext is $P[v]\ot W_1$, which is multiplied by higher $z$'s. Since $B_k$'s are multiplied by higher $z$'s, but $A_k$'s are not, this explains why $z_1$ is in $B_1$ when $p$ is odd, but not when $p=2$. The reason for the split in Theorem \ref{E2} is the difference in the submodules $N$. Its second class is $y_0^{p-1}Q_1y_0$ in each.  Applying $Q_1$ yields $y_0^{p-2}(Q_1y_0)^2$. This is 0 when $p$ is odd, but not when $p=2$. The reason that the portion of Ext corresponding to $N$ is not multiplied by higher $z$'s is that it gives part of the $Q_0$-homology, and this is not multiplied by higher $z$'s.

\bigskip

We close this section with enumeration of the unimportant $\zt$-classes in $ku^*(K_2)$ when $p=2$.

\begin{center}
\textbf{More on the $E_1$-free part when $p=2$}
\end{center}

If we compute the $\Ext_{E_1}(\zt,F)$ for the $E_1$ free part of $H^*K_2$, we
just get a $\zt$ corresponding to the top element for each copy of $E_1$.
If we find the
Poincar\'e
series (PS) for the free part, all we have to do to get the PS for these elements is multiply
by $\frac{x^4}{(1+x)(1+x^3)}$.
The
Poincar\'e
series for free part is obtained by subtracting the
PS
for the non-free part of Proposition \ref{T} from that of $H^*K_2$.
This is:

$$
\prod_{k\ge 0} \frac{1}{(1-x^{2^k+1})}
-\frac{1}{(1-x^4)}
\bigl(1 + x^5+x^7+x^8+x^9+x^{10}\bigr)
$$
$$
-\frac{1}{(1-x^2)(1-x^4)}
\bigl(
\bigoplus_{j\ge 4}
\bigl(
x^{2^j+1}(1+x^9)(1+x)(1-x^{2j-6}) \prod_{k\ge j} (1+x^{2^{k+1}+2})
\bigr)
\bigr)
$$

The first term is the PS for $H^*K_2$.  The second is the PS for
$P[u_2^2]\otimes (\langle 1 \rangle \oplus N)$.
The last term is more complicated but does the $S$ and $R$ terms.
The $(1-x^4)$ in the denominator is for the $P[u_2^2]$.  The
$x^9$ is the shift that takes $R$ to $S$.  The $(1+x)$ is because they
are $Q_0$ free.  The $x^{2^j+1} (1-x^{2j-6})/(1-x^2)$ is for the odd part of
$M_j$ and the remainder is for $\L$.

This is easy to put into a computer and calculate.  For example, the
number of free generators in degree 79 is 245.

\section{Differentials in the ASS of $ku^*(K_2)$} \label{difflsec}
The main theorem of this section determines the differentials in the ASS for $ku^*(K_2)$.

\begin{thm}\label{diffl}
The differentials in the spectral sequence whose $E_2$-term was given in Theorem \ref{E2} are as follows. All $v$-towers are involved, either as source or target, in exactly one of these. Here  $M$ refers to any monomial (possibly $=1$) in the specified  algebra. Recall that $\L_j=TP_p[z_i:\,i\ge j]$, which is an exterior algebra if $p=2$. Also, recall $y_t=y_1^{p^{t-1}}$. We give reference numbers to the differentials when $p$ is odd, but references to these also apply to the corresponding differential when $p=2$, as the proofs are extremely similar.

First with $p=2$.
\begin{eqnarray*} d_{\nu(i)+2}(y_1^i)&=&h_0^{\nu(i)}v^2q y_1^{i-1},\ i\ge1;\\
d_{\nu(i)+2}(y_1^iz_jM)&=&v^{\nu(i)+2}q y_1^{i-1}z_{j-\nu(i),j}M,\\
&&j\ge \nu(i)+2,\ M\in\L_j;\\
d_{2^t-t}(h_0^{t-2}v^2q y_1^{2^{t-1}-1}M)&=&v^{2^t}z_tM,\\
&&t\ge2,\ M\in P[y_t];\\
d_{2^t-t}(q y_1^{2^{t-1}-1}z_{j-(t-2),j}M)&=&v^{2^t-t}z_tz_jM,\\
&&j\ge t\ge2,\ M\in P[y_t]\ot \L_{j+1}.\end{eqnarray*}

Now with $p$ odd.
\begin{eqnarray}d_{\nu(i)+2}(y_1^i)&=&h_0^{\nu(i)+1}vq y_1^{i-1},\ i\ge1;\label{1}\\
d_{\nu(i)+2}(y_1^iz_jM)&=&v^{\nu(i)+2}q y_1^{i-1}z_{j-\nu(i)-1,j}M,\label{2}\\
&&j\ge\nu(i)+2,\ M\in \L_j;\nonumber\\
d_{p^t-t}(h_0^{t-1}vq y_1^{p^{t-1}-1}M)&=&v^{p^t}z_tM,\label{3}\\
&&t\ge1,\ M\in P[y_t];\nonumber\\
d_{p^t-t}(q y_1^{p^{t-1}-1}z_{j-(t-1),j}M)&=&v^{p^t-t}z_tz_jM,\label{4}\\
&&j\ge t\ge1,\ M\in P[y_t]\ot TP_{p-1}[z_j]\ot \L_{j+1}.\nonumber\end{eqnarray}

\end{thm}

The proof occupies the rest of this section, except that at the end of the section we explain briefly how this leads to our description of $ku^*(K_2)$ in Section \ref{intro}, except for the exotic extensions.

By \cite[Theorem A]{Tam}, $Q_jQ_0\io_2$ is in the image from $BP^*(K_2)$, and hence must be a permanent cycle in our ASS. Thus by (\ref{zj}),
 $z_j$ is a permanent cycle, and so (\ref{2}) follows from (\ref{1}), and (\ref{4}) follows from (\ref{3}), using $pz_{i,\ell}=vz_{i-1,\ell}$, as noted in \ref{Sdef}.

The differentials (\ref{1}) follow from the result of \cite{Br} that $H^{2pi+1}(K_2;\Z)\approx\bz/p^{\nu(i)+2}\oplus\bigoplus\zp$. See also \cite[Proposition 1.3.5]{Cle} when $p=2$. The ASS converging to $H^*(K_2;\Z)$ has $E_2=\ext_{A_0}(\zt,H^*K_2)$, where $A_0=\langle 1,Q_0\rangle$. We depict this $E_2$ similarly to our ASS for $ku^*(K_2)$. It has an $h_0$-tower for each element of $H_*(H^*K_2,Q_0)$, which was described in Lemma \ref{Q0}. These come in pairs in grading $2pi$ and $2pi+1$ corresponding to $y_1^i$ and $y_1^{i-1}y_0^{p-1}u_0$. In order to get the $\Z/p^{\nu(i)+2}$, there must be a $d_{\nu(i)+2}$-differential, as pictured on the right hand side of Figure \ref{Brfig}.

Similarly to Figures \ref{fig3} and \ref{fig4},  we have, for $p=2$ and $i\ge1$, an $h_0$-tower in the ASS for $ku^*(K_2)$ arising from $G^{4i+1,2}$, called either $h_0^2y_1^{i-1}x_5$ or $v^2y_1^{i-1}q$. There is also an $h_0$-tower arising from $y_1^i\in G^{4i,0}$. The classes $y_1$ and $x_5$ correspond to cohomology classes $u_2^2$ and $u_5+u_2u_3$. Under the morphism $ku^*(K_2)\to H^*(K_2;\Z)$, these towers map across, as suggested in Figure \ref{Brfig}. We deduce the $d_{\nu(i)+2}$-differential claimed in (\ref{1}), promulgated by the action of $v$. Note that  $x_9=q$.

\bigskip
\begin{minipage}{6in}
\begin{fig}\label{Brfig}

{\bf $ ku^*(K_2)\to H^*(K_2;\Z)$}

\begin{center}

\begin{\tz}[scale=.55]
\draw (-1,0) -- (3,0);
\draw (.4,2.2) -- (0,2) -- (0,7);
\draw (2.4,.2) -- (2,0) -- (2,7);
\draw [->] (2,0) -- (.5,4);
\draw (7,0) -- (11,0);
\draw (8,0) -- (8,7);
\draw (10,0) -- (10,7);
\draw [->] (10,0) -- (8.5,4);
\node at (2,-.5) {$2pi$};
\node at (0,-.5) {$2pi+1$};
\node at (10,-.5) {$2pi$};
\node at (8,-.5) {$2pi+1$};
\node at (1,-1.3) {$ku^*(K_2)$};
\node at (9,-1.3) {$H^*(K_2;\Z)$};
\node at (5,2) {$\longrightarrow$};
\end{\tz}
\end{center}
\end{fig}
\end{minipage}

\bigskip
The situation when $p$ is odd is extremely similar, using Figure \ref{Nfig}. The difference is that the $h_0$-tower in $2pi+1$ in the $ku^*$ ASS starts in filtration 1 rather than 2. Its generator can be called $vy_1^{i-1}q$.

In Figure \ref{low}, we depict many of the differentials asserted in Theorem \ref{diffl} in grading $\le36$ when $p=2$. Regarding the third (final) summand in Theorem \ref{E2}, which is $P[y_1]\ot A_1$ when $p=2$, we have included $y_1A_1$, $y_1^3A_1$, and $y_1^5A_1$. Not included are the portions involving (\ref{1}) and (\ref{2}) when $i$ is odd, as this portion self-annihilates. What is shown is (\ref{1}) for $i=2$, 4, and 6, (\ref{3}) for $(t,k)=(1,0)$, $(1,1)$, $(1,2)$, and $(2,0)$, and (\ref{4}) with $t=1$, $k=0$, and $j=4$.

\bigskip

\tikzset{
  testpic2/.pic=
{\draw (2,0) -- (20.6,9.3);
\draw [color=blue] (-1,0) -- (29,0);
\draw (18,0) -- (28.6,5.3);
\draw (27,2) -- (29.6,3.3);
\draw [color=red] (27,2) -- (26,4);
\draw [color=red] (29,3) -- (28,5);
\node at (18,0) {\lb};
\node at (19.6,.8) {\lb};
\node at (22,2) {\lb};
\node at (24,3) {\lb};
\node at (26,4) {\lb};
\node at (28,5) {\lb};
\node at (27,2) {\lb};
\node at (29,3) {\lb};
\node at (2,0) {\lb};
\node at (4,1) {\lb};
\node at (6,2) {\lb};
\node at (8,3) {\lb};
\node at (10,4) {\lb};
\node at (12,5) {\lb};
\node at (14,6) {\lb};
\node at (16,7) {\lb};
\node at (18,8) {\lb};
\node at (20,9) {\lb};
\node at (0,0) {\lb};
\node at (2,1) {\lb};
\node at (4,2) {\lb};
\node at (6,3) {\lb};
\node at (8,4) {\lb};
\draw (0,0) -- (8.6,4.3);
\draw (2,0) -- (2,1);
\draw (4,1) -- (4,2);
\draw (6,2) -- (6,3);
\draw (8,3) -- (8,4);
\draw (5,0) -- (9.6,2.3);
\draw [color=red] (5,0) -- (4,2);
\draw [color=red] (7,1) -- (6,3);
\draw [color=red] (9,2) -- (8,4);
\node at (5,0) {\lb};
\node at (7,1) {\lb};
\node at (9,2) {\lb};
\node at (0,-.5) {$36$};
\node at (4,-.5) {$32$};
\node at (8,-.5) {$28$};
\node at (12,-.9) {$24$};
\node at (16,-.5) {$20$};
\node at (20,-.5) {$y_1^4$};
\node at (20,-.9) {$16$};
\node at (24,-.5) {$12$};
\node at (28,-.5) {$8$};
\node at (6,0) {\lb};
\node at (8,0) {\lb};
\node at (14,0) {\lb};
\node at (16,0) {\lb};
\node at (22,0) {\lb};
\node at (24,0) {\lb};
\node at (8,1) {\lb};
\node at (16,1) {\lb};
\node at (24,1) {\lb};
\draw (6,0) -- (8,1) -- (8,0);
\draw (14,0) -- (16,1) -- (16,0);
\draw (22,0) -- (24,1) -- (24,0);
\node at (22,-.4) {$y_1z_1$};
\node at (0,.4) {$z_2^2$};
\node at  (1.7,.3) {$z_3$};
\node at (2.4,-.4) {$y_1^4z_2$};
\node at (5.3,-.4) {$y_1qz_2$};
\node at (18,-.4) {$z_2$};
\node at (26.3,2) {$v^2y_1q$};
\draw (10,0) -- (20.6,5.3);
\draw (19,5.3) -- (19,2) -- (21,3) -- (21,4) -- (19,3);
\draw [color=red] (19,2) -- (18,4);
\draw [color=red] (21,3) -- (20,5);
\draw [color=red] (19,3) -- (18,8);
\draw [color=red] (21,4) -- (20,9);
\node at (10,-.4) {$y_1^2z_2$};
\node at (18.3,2) {$v^2y_1^3q$};
\draw (21,3) -- (21.6,3.3);
\draw (21,4) -- (21.6,4.3);
\draw (28,0) -- (28,1.3);
\draw (27,2) -- (27,4.3);
\draw [color=red] (28,0) -- (27,3);
\draw [color=red] (28,1) -- (27,4);
\node at (28.3,.3) {$y_1^2$};
\node at (28,0) {\lb};
\node at (28,1) {\lb};
\node at (27,3) {\lb};
\node at (27,4) {\lb};
\draw (20,0) -- (20,1.3);
\node at (20,0) {\lb};
\node at (20,1.2) {\lb};
\draw (20,0) -- (20.6,.3);
\draw (20,1.2) -- (20.6,1.5);
\draw [color=red] (20,0) -- (19,4);
\draw [color=red] (20,1.2) -- (19,5);
\node at (10,0) {\lb};
\node at (12,1) {\lb};
\node at (14,2) {\lb};
\node at (16,3) {\lb};
\node at (18,4) {\lb};
\node at (20,5) {\lb};
\node at (19,2) {\lb};
\node at (19,3) {\lb};
\node at (19,4) {\lb};
\draw (19,4) -- (19.6,4.3);
\draw (19,5) -- (19.6,5.3);
\node at (19,5) {\lb};
\node at (21,3) {\lb};
\node at (21,4) {\lb};
\draw (21,4) -- (21,4.3);

\node at (2.3,0) {\lb};
\node at (4.3,1) {\lb};
\node at (6.3,2) {\lb};
\node at (8.3,3) {\lb};
\node at (10.3,4) {\lb};
\node at (12.3,5) {\lb};
\draw (2.3,0) -- (12.9,5.3);
\draw (11,2) -- (13.6,3.3);
\node at (11,2) {\lb};
\node at (13,3) {\lb};
\draw [color=red] (11,2) -- (10.3,4);
\draw [color=red] (13,3) -- (12.3,5);
\draw (12,0) -- (12,.5);
\node at (12,0) {\lb};
\node at (12,-.4) {$y_1^6$};
\draw (11,2) -- (11,3.3);
\node at (11,3) {\lb};
\draw [color=red] (12,0) -- (11,3);
\node at (11,1.6) {$v^2y_1^5q$};
\draw [color=red] (19,3) -- (18,8);
\draw [color=red] (21,4) -- (20,9);
\draw (12,0) -- (12.6,.3);
\draw (11,3) -- (11.6,3.3);
}}
\bigskip
\begin{minipage}{6in}
\begin{fig}\label{low}

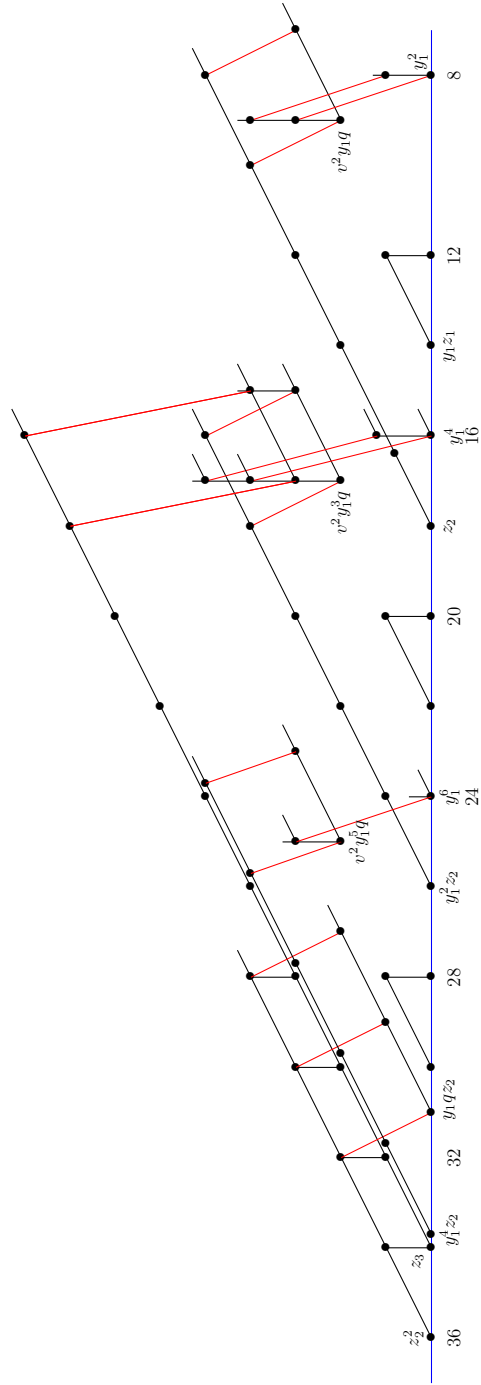
{\bf Some differentials with $p=2$}
\begin{center}
\begin{tikzpicture}
  \pic[rotate=90,scale=.6,transform shape] {testpic2};
\end{tikzpicture}
\end{center}
\end{fig}
\end{minipage}

\bigskip

In order to establish the remaining differentials, we will need the following description of $k(1)^*(K_2)$, which is proved in \cite{DRW}.
We shift by 1 the subscripts of the classes $z_j$ and $w_j$ used there.
The formulas for $r(j)$ and $r'(j)$ are as in \cite{DRW}. We recapitulate some of their properties. Those stated here but not there are easily proved by induction.
\begin{prop} \cite{DRW} \label{rprop}For $j\ge0$, $z_j$ is the reduction of the class in $ku^*(K_2)$ and satisfies $|z_j|=2(p^{j+1}+1)$. The classes $w_j$ satisfy $|w_1|=2p^2+1$, $|w_2|=2p^3-2p^2+6p-3$, and $w_{j+2}=y_j^{p-1}w_jz_{j+1}^{p-1}$. The integers $r(j)$ and $r'(j)$ satisfy the following properties.
\begin{eqnarray}&&r(0)=1,\ r(1)=p,\ r(j+2)=r(j)+p^{j+1}(p-1)+1;\label{rrec}\\
&&r'(0)=p-1,\ r'(1)=p^2-p,\nonumber\\
&&r'(j+2)=r'(j)+p^{j+2}(p-1)-1,\label{rprec}\\
&&r(j)-r'(j-1)=j,\label{r1}\\
&&r(j)+r'(j)=p^{j+1},\label{r2}\\
&&r(j+2)+r'(j)=p^{j+2}+1,\label{r3}\\
&&(p-1)(r(j-1)+j-1)<p^j,\label{r4}\\
&&p^{j+1}-p^j\le r'(j)<p^{j+1}-p^{j-1}.\label{r5}
\end{eqnarray}
\end{prop}
\begin{thm}\label{DRWthm}\cite{DRW} For any $p$, $k(1)^*(K_2)$ is a trivial $k(1)^*$-module plus
\begin{eqnarray*}&&\bigoplus_{j>0}TP_{r(j)}[v]\ot P[y_{j+1}]\ot TP_{p-1}[y_j]\ot \Ebar[w_{j}]\ot E[w_{j+1}]\ot \L_{j+1}\\
&\oplus&\bigoplus_{j\ge1}TP_{r'(j-1)}[v]\ot P[y_{j}]\ot E[w_{j}]\ot\overline{TP}_p[z_{j}]\ot \L_{j+1}\\
&\oplus&P[y_1]\ot\biggl(\Ebar[y_0^{p-1}z_0]\oplus\begin{cases}\Ebar[z_1]&p=2\\ 0&p\text{ odd}\end{cases}\biggr)\oplus\bigoplus_{j\ge1} P[y_1]\ot E[q ]\ot \Ebar[z_j^p]\ot \L_{j+1}. \end{eqnarray*}\end{thm}
\ni The last line was not discussed in \cite{DRW}; it is from free $E[Q_1]$ summands which are not part of free $E_1$ summands, and plays a very important role.

Now we continue the proof of Theorem \ref{diffl}.
We have already proved (\ref{1}) and (\ref{2}). As already noted, the $z_j$'s are infinite cycles by \cite{Tam}, and so the differentials in (\ref{4}) are implied as soon as the corresponding differential in (\ref{3}) is proved.

 As a warmup, we consider the cases $t=2$ and 3 of (\ref{3}) when $p=2$. We make extensive use of the exact sequence (\ref{LES}). Referring to Figure \ref{low} is useful.

In even gradings $\le14$, $k(1)^*(K_2)=0$ in positive filtration, by Theorem \ref{DRWthm}. Thus the map $ku^*(K_2)\to k(1)^*(K_2)$ implies that in the ASS for $ku^*(K_2)$,
  $v^sz_2$ must be hit by a differential or divisible by 2 for $s\ge2$. In grading $<8$, there is nothing that can divide it, and the only odd-grading $v$-tower in that range is on $v^2y_1q$. Thus $d_2(v^2y_1q)=v^4z_2$, the case $t=2$, $M=1$ of (\ref{3}). Since $d_2(y_1^{2k})=0$ by (\ref{1}), the case $t=2$ of (\ref{3}) follows for any $M$ by the derivation property. An analogous argument does not work at the odd primes.

Similarly $v^sz_3$ must be hit or divisible for  $s\ge4$, and examination of options in Figure \ref{low} shows that we must have $d_5(h_0v^2y_1^3q)=v^8z_3$, preceded by extensions. Since $d_5(y_1^8)=h_0^3v^2y_1^7q$, we deduce the case $t=3$, $M\in P[y_1^8]$ of (\ref{3}) using the derivation property, (\ref{x9z4}) and $h_0z_2=0$. We do not have {\it a priori} knowledge that $y_1^4z_3$ is a permanent cycle in the ASS of $ku^*(K_2)$. However, if it supported a nonzero differential, then the tower of $v$-height 4 on $y_1^4z_3$ in the ASS of $k(1)^*(K_2)$ would have to map to $v^tC$ for $0\le t\le3$ for some $C$ in positive filtration in grading 51 in the ASS of $ku^*(K_2)$. Then $v^4C$ must be  $d_r(B)$ with $r\ge5$ and $B$ in filtration 0 in grading 42. ($B$ cannot have higher filtration since everything is $v$-towers, and $v^3C$ cannot be hit.) But the only possible $B$ is $y_1^6z_2$, and we already know that $v^4y_1^6z_2\in\im(d_4)$. (Ordinarily this would not preclude the possibility of $B$ supporting a differential, but it does since everything is $v$-towers.) Thus $y_1^4z_3$ is a permanent cycle, and consideration of its image in $k(1)^*(K_2)$ implies that $v^sy_1^4z_3$ is hit by a differential for some $s\ge4$. The only element in odd grading $<42$ not yet accounted for is $h_0v^2y_1^7q$ in grading 33, and so this must be the source of the differential. This is the case $t=3$, $M=y_1^4$ of (\ref{3}). The validity for all  $M=y_1^{8i+4}$ (and $t=3$) now follows similarly to what we did for $M=y_1^{8i}$ at the beginning of this paragraph.

Now we switch our attention to the odd primes. The situation when $p=2$ is extremely similar. We want to prove the following version of (\ref{3}).
\begin{equation}\label{3n}d_{p^t-t}(h_0^{t-1}vqy_1^{(i+1)p^{t-1}-1})=v^{p^t}y_1^{ip^{t-1}}z_t.\end{equation}

Now we work toward proving this.
We illustrate with $p=5$, but it should be clear how it generalizes to an arbitrary prime. One new thing is the Divisibility Criterion as invoked in \cite{DRW}. Each mod $(p-1)$ value of $i$ can be considered separately. We will consider (\ref{3n}) with $p=5$ and $i=4\ell$; other congruences follow similarly. We index the differential (\ref{3n}) by $(\ell,t)$. We write $T$ (for vertical Tower) for the class $h_0^{t-1}vq y_1^{(4\ell+1)5^{t-1}-1}$, and $M$ (for Monomial) is $y_1^{4\ell5^{t-1}}z_t$. We will often afflict $T$ and $M$ with the parameters $(\ell,t)$. We write $|T|$ for $\frac12(|T|+1)$. The $\frac12$ avoids extraneous factors of 2 that always cancel out. The $+1$ is so that this indicates the grading (times $\frac12$) of the class that it hits. $|M|$ denotes $\frac12$ times the grading of $M$, and $M'$ equals $\frac12$ times the grading of $v^hM$, where $h$ is the $v$-height of $M$ in $k(1)^*(K_2)$. We wish to show that the differentials {\it must be} as claimed.

There are three types of constraints on the differentials involving these classes.
Constraint C1 is that if $T\to M$ (by which we mean that a certain $T$ class supports a differential hitting $v^iM$ for some $i$ and a certain monomial $M$), then $|T|\le M'$. (This says that the $v$-tower on $M$ cannot be hit while its image in $k(1)^*(K_2)$ is nonzero.) Constraint C2 says that if $T(5\ell+1,t-1)\to M_1$ and $T(\ell,t)\to M_2$, then $|M_2|>|M_1|$. Since $|T(5\ell+1,t-1)|=|T(\ell,t)|$, this says that as you move up an $h_0$ tower\footnote{Note that $h_0T(5\ell+1,t-1)=T(\ell,t)$.}, differentials must get longer (unless they are hitting into an $h_0$ tower, which is not the case here.) Constraint C3  says that if $T_2\to M_1$, then there exists $M_3$ such that $|M_1|\ge|M_3|\ge M_1'$ and either
$$M_3'\le|T_2|$$
or
$$T_3\to M_3\text{ has already been proved, and } |T_3|\le |T_2|.$$  The reason for C3 is that there must be extensions into the $M_1$-tower from grading $M_1'$ to $|T_2|+4$. The nonzero classes on the $v$-tower (on $M_3$) supporting the extensions must go to at least $|T_2|+4$, and it has nonzero classes at least to $M_3'+4$, and if $T_3\to M_3$ was already proved, it has nonzero classes to $|T_3|+4$. Note that we are saying that the $v$-tower on $M_1$ maps to 0 in $k(1)^*(K_2)$ once we get to grading $M_1'$ (and hence in gradings $\le M_1'$ it is either hit by differentials or is divisible by $p$). There might be classes of higher filtration in $k(1)^*(K_2)$ to which it could map, but, if so, we can modify the generator of the $M_1$ tower by the class on the tower sitting above it. Also note that it is possible that extensions from the tower $M_3$ don't start from the generator, if there are $h_0$-extensions on the tower for awhile.  See Figure \ref{figM3}. There is an exception to the C3 requirement for $T(\ell,1)\to M(\ell,1)$. Here the extension into $v^4y_1^{4\ell}z_1$ is obtained from the special class $y_1^{4\ell}y_0^4z_0$.

\bigskip

\begin{minipage}{6in}
\begin{fig}\label{figM3}

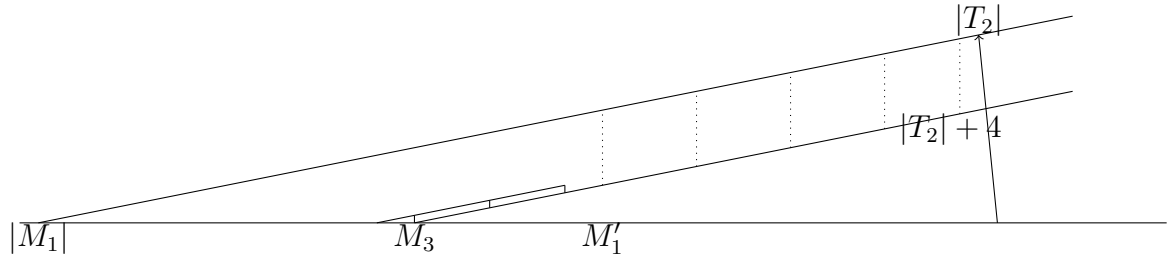
{\bf The role of $M_3$}

\begin{center}

\begin{\tz}[scale=.25]
\draw (0,0) -- (55,11);
\draw (-1,0) -- (60,0);
\draw (20,0) -- (55,7);
\draw (18,0) -- (28,2);
\draw (20,0) -- (20,.4);
\draw (24,.8) -- (24,1.2);
\draw (28,1.6) -- (28,2);
\draw [->] (51,0) -- (50,10);
\draw [dotted] (30,2) -- (30,6);
\draw [dotted] (49,5.8) -- (49,9.8);
\draw [dotted] (35,3) -- (35,7);
\draw [dotted] (40,4) -- (40,8);
\draw [dotted] (45,5) -- (45,9);
\node at (0,-.8) {$|M_1|$};
\node at (20,-.8) {$M_3$};
\node at (50,10.8) {$|T_2|$};
\node at (48.5,5) {$|T_2|+4$};
\node at (30,-.8) {$M_1'$};
\end{\tz}\end{center}
\end{fig}
\end{minipage}

\bigskip
With the above conventions, we have $|T|=5^t(4\ell+1)+1$, $|M|=5^t(4\ell+5)+1$, and $M'=|M|-4r'(t-1)$, where $4r'(t-1)$ has the values 16, 80, 412, and 2076 for $t=1$, 2, 3, and 4. Increasing from $t$ to $t+2$ increases this by $4^2\cdot5^{t+1}-4$. We consider the cases in order of increasing $|M|$ and, for equal values of $|M|$, increasing $\ell$. We tabulate a representative sample in Table \ref{tabl}. We omit listing values of $\ell\equiv3,4$ mod 5 because they behave similarly to $\ell\equiv2$.

\vfill\eject
\begin{table}[h]
\caption{Cases in order}
\label{tabl}
\renewcommand{\arraystretch}{1.15}

\begin{tabular}{cc|ccc|ccc|ccc}
$\ell$&$t$&$|T|$&$|M|$&$M'$&\qquad\qquad&$\ell$&$t$&$|T|$&$|M|$&$M'$\\
\hline
0&1&6&26&10&&36&1&726&746&730\\
1&1&26&46&30&&37&1&746&766&750\\
2&1&46&66&50&&7&2&726&826&746\\
0&2&26&126&46&&40&1&806&826&810\\
5&1&106&126&110&&41&1&826&846&830\\
6&1&126&146&130&&42&1&846&866&850\\
7&1&146&166&150&&8&2&826&926&846\\
1&2&126&226&146&&45&1&906&926&910\\
10&1&206&226&210&&46&1&926&946&930\\
11&1&226&246&230&&47&1&946&966&950\\
12&1&246&266&250&&9&2&926&1026&946\\
2&2&226&326&246&&50&1&1006&1026&1010\\
15&1&306&326&310&&51&1&1026&1046&1030\\
16&1&326&346&330&&52&1&1046&1066&1050\\
17&1&346&366&350&&1&3&626&1126&714\\
3&2&326&426&346&&10&2&1026&1126&1046\\
20&1&406&426&410&&55&1&1106&1126&1110\\
21&1&426&446&430&&56&1&1126&1146&1130\\
22&1&446&466&450&&57&1&1146&1166&1150\\
4&2&426&526&446&&11&2&1126&1226&1146\\
25&1&506&526&510&&60&1&1206&1226&1210\\
26&1&526&546&530&&61&1&1226&1246&1230\\
27&1&546&566&550&&62&1&1246&1266&1250\\
0&3&126&626&214&&&&$\vdots$&&\\
5&2&526&626&546&&154&1&3086&3106&3090\\
30&1&606&626&610&&0&4&626&3126&1050\\
31&1&626&646&630&&5&3&2626&3126&2714\\
32&1&646&666&650&&30&2&3026&3126&3046\\
6&2&626&726&646&&155&1&3106&3126&3110\\
35&1&706&726&710&&156&1&3126&3146&3130
\end{tabular}
\end{table}

\bigskip
Before presenting a general argument, we illustrate with an example, starting with $M_1=M(1,3)$. We will see that it builds a chart which is $y_1^{100}$ times Figure \ref{oddchart}. In Table \ref{tabl}, we have $|M_1|=1126$. Its $v$-tower is truncated at height $p^3=125$ by a differential on $T(1,3)$, with $|T(1,3)|=626$, using our grading conventions. Playing the role of $M_3$ is $M(6,2)$ with $|M_3|=726$. We have $M_1'=714$. It is $v^3M_3$ which supports the extension in ''grading`` 714. Note that for $0\le i\le 2$, $h_0v^iM_3\ne0$, and so $p\cdot v^iM_3$ is not a $v$-multiple of $M_1$. (In Figure \ref{oddchart}, the class $y_2^{p-1}z_2$ corresponds to $M_3$.) From Table \ref{tabl}, we see that $M_3'=646$, which means that in ``grading'' $\le646$, the $v$-tower on $M_3$ is either hit by a differential or divisible by $p$. Table \ref{tabl} says it is hit by a differential in 626. In ``gradings'' from 646 to 630, it is divisible by $p$. It has its own, distinct, $M_3$ class, namely $M(31,1)$. In Figure \ref{oddchart}, this latter class corresponds to $y_1^{p-1}y_2^{p-1}z_1$.

\medskip
Now we start the proof. We begin with a lemma.
\begin{lem} For $M=M(\ell',t')$ with $|M(5\ell+1,t-1)|<|M|<|M(\ell,t)|$, we have $t'<t$, $|T(\ell,t)|<|T(\ell',t')|$, and $|M(5\ell+1,t-1)|<M'$.\end{lem}
\begin{proof} The given inequalities quickly force $t'<t$. The inequality $|T(\ell,t)|<|T(\ell',t')|$ follows immediately. Finally, the given inequalities prevent $M'\le |M(5\ell+1,t-1)|$. \end{proof}

To prove the differentials, we use induction on our ordering of the $M$'s.  If the differentials are not as posed, consider the smallest $|M|$ such that $T(\ell,t)\to M$ with $M\ne M(\ell,t)$.

We cannot have $|M|>|M(\ell,t)|$, because $|M(\ell,t)|$ would contradict the minimality of $|M|$. We cannot have $|M|\le |M(5\ell+1,t-1)|$ by constraint C2.

If $|M(5\ell+1,t-1)|<|M|<|M(\ell,t)|$, by constraint C3 and the lemma, we must have $M_3$ with $|M(5\ell+1,t-1)|<M'\le |M_3|<|M|$. Because $|M_3|<|M|$, we know $T_3\to M_3$ by induction. From the lemma, we get $|T(\ell,t)|<|T_3|$, but that contradicts constraint C3.

We must have $T(\ell,t)\to M(\ell,t)$, and $M(5\ell+1,t-1)$ is eligible for our $M_3$.
 This completes most of the proof of (\ref{3n}) and hence of Theorem \ref{diffl}.

\bigskip
Underlying the above analysis has been an assumption that the $M$-classes are always hit by $T$-classes. We show now that it could not have occurred that an $M$-class supported a differential. Assume that $M=y_1^{ip^{t-1}}z_t$ is the $M$-class of lowest grading which supports a differential. We now revert to letting $|x|$ denote the actual grading of a class $x$, not divided by 2.

In $k(1)^*(K_2)$, $M$ supports a $v$-tower of $v$-height $r'(t-1)$ by \ref{DRWthm}. We will show at the end of the proof that there is a number $\Delta\le t$ such that $v^iM$ maps nontrivially to $ku^{*+1}(K_2)$ if and only if $i\le r'(t-1)-\Delta$. (Usually $\Delta=1$.) The image of $M$ in $ku^{|M|+1}(K_2)$ is a class $C$ of positive filtration such that $v^{r'(t-1)-\Delta}C\ne0$ and $v^{r'(t-1)-\Delta+1}C=0\in ku^*(K_2)$, so there must be a differential in the ASS of $ku^*(K_2)$ from a filtration-0 class hitting a class of filtration $\ge r'(t-1)-\Delta+2$ in grading $|M|+1-2(p-1)(r'(t-1)-\Delta+1)$. (The reason that the differential must start from filtration 0 is that in even gradings, $E_2$ consists entirely of $v$-towers starting in filtration 0.) This differential cannot come from another such $M$ because of our lowest-grading assumption. It cannot come from a product of one or more $z$'s times one of these $M$'s because $z$'s are infinite cycles. We must rule out the possibility that this differential is one of type (\ref{2}). They are distinguished by having the smallest $z$-subscript at least 2 greater than the $p$-exponent of the exponent of $y_1$.

The differential to $C$ has subscript $\ge r'(t-1)-\Delta+2$, and so the class in (\ref{2}) would be $y_1^{\ell p^{r'(t-1)-\Delta}}Z$ for some positive integer $\ell$, where $Z$ is a product of $z_j$'s with $j\ge r'(t-1)-\Delta+2$, and each $j$ appears at most $p-1$ times, except that the smallest $j$ might appear $p$ times. Equating this grading with $|M|-2(p-1)(r'(t-1)-\Delta+1)$, and cancelling a common factor 2 from all terms yields
\begin{equation}\ell p^{r'(t-1)-\Delta+1} + \sum_j (p^{j+1}+1) =ip^t+p^{t+1}+1-(p-1)(r'(t-1)-\Delta+1).\label{ppt}\end{equation}
Using (\ref{r2}) and (\ref{r4}) and $\Delta\le t$, the right hand side of (\ref{ppt}) equals $p^t(i+1)+(p-1)(r(t-1)+\Delta-1)+1\equiv (p-1)(r(t-1)+\Delta-1)+1$ mod $p^t$, with $(p-1)(r(t-1)+\Delta-1)+1\le p^t$ (strict if $t>2$). Since $r'(t-1)-\Delta>t$, this implies that the $\ds\sum_j$ on the left hand side of (\ref{ppt}) must contain at least $(p-1)(r(t-1)+\Delta-1)+1$ summands. We obtain
\begin{eqnarray*}&&\sum p^{j}\ge p\cdot p^{r'(t-1)-\Delta+2}+(p-1)(p^{r'(t-1)-\Delta+3}+\cdots+ p^{r'(t-1)+r(t-1)})\\
&=&p^{r'(t-1)+r(t-1)+1}=p^{p^t+1},\end{eqnarray*}
so $\sum p^{j+1}\ge p^{p^t+2}$, and hence $p^t(i+1)>p^{p^t+2}$. Thus $i\ge p^{p^t-t+2}>p^{p^t-2t}$.

Since $d_{p^t-t+1}(y_1^{p^{p^t-t-1}})$ is defined, \begin{equation}\label{pp2t}d_r(y_1^{p^{p^t-t-1}})=0\text{ for }r\le p^t-t,\end{equation}
and by the lowest-grading assumption, $ d_{p^t-t}(h_0^{t-1}vqy_1^{(i-p^{p^t-2t}+1)p^{t-1}-1})=v^{p^t}y_1^{(i-p^{p^t-2t})p^{t-1}}z_t$ and $y_1^{(i-p^{p^t-2t})p^{t-1}}z_t$ is a permanent cycle. Since
$$y_1^{ip^{t-1}}z_t=y_1^{(i-p^{p^t-2t})p^{t-1}}z_t\cdot y_1^{p^{p^t-t-1}},$$
we deduce that $y_1^{ip^{t-1}}z_t$ survives to $E_{p^t-t}$ and (\ref{3n}),
 using the derivation property of differentials.

Now we consider the need for $\Delta$ in the above argument. The worry is that maybe part of the $v$-tower on $M$ in $k(1)^*(K_2)$ might be in the image from $ku^*(K_2)$, due to a filtration jump from a lower tower, as sketched in Figure \ref{unw}, so that only a smaller part of the $M$-tower in $k(1)^*(K_2)$ maps to $ku^{*+1}(K_2)$.

\bigskip

\begin{minipage}{6in}
\begin{fig}\label{unw}

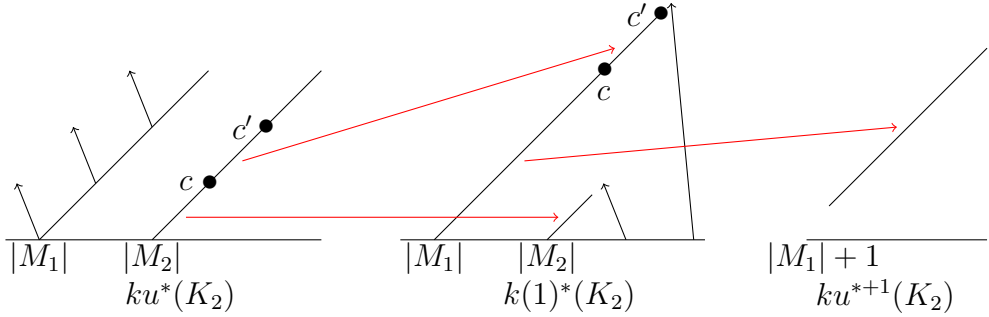
{\bf An unwanted possibility}

\begin{center}

\begin{\tz}[scale=.15]
\draw (-3,0) -- (25,0);
\draw (0,0) -- (15,15);
\draw (10,0) -- (25,15);
\draw (32,0) -- (59,0);
\draw (35,0) -- (55,20);
\draw (45,0) -- (49,4);
\draw (68,0) -- (84,0);
\draw (70,3) -- (84,17);
\draw [->] (0,0) -- (-2,5);
\draw [->] (5,5) -- (3,10);
\draw [->] (10,10) -- (8,15);
\node at (15,5) {\lb};
\node at (13,5){$c$};
\node at (20,10) {\lb};
\node at (18,10) {$c'$};
\node at (50,15) {\lb};
\node at (50,13) {$c$};
\node at (55,20) {\lb};
\node at (53,20) {$c'$};
\draw [->] (58,0) -- (56,21);
\draw [->] (52,0) -- (50,5);
\draw [->] [color=red] (13,2) -- (46,2);
\draw [->] [color=red] (18,7) -- (51,17);
\draw [->] [color=red] (43,7) -- (76,10);
\node at (0,-1.6) {$|M_1|$};
\node at (10,-1.6) {$|M_2|$};
\node at (35,-1.6) {$|M_1|$};
\node at (45,-1.6) {$|M_2|$};
\node at (69.5,-1.6) {$|M_1|+1$};
\node at (12.5,-5) {$ku^*(K_2)$};
\node at (47,-5) {$k(1)^*(K_2)$};
\node at (75,-5) {$ku^{*+1}(K_2)$};

\end{\tz}
\end{center}
\end{fig}
\end{minipage}

\bigskip

The monomials $M_\eps=y_{t_\eps}^{i_\eps}z_{t_\eps}$ ($\eps=1,2$) have $|M_\eps|=2(p^{t_\eps}(i_\eps+p)+1)$ and are truncated in $k(1)^*(K_2)$ in grading $M_\eps'=|M_\eps|-2(p-1)r'(t_\eps-1)$. In $ku^*(K_2)$, $M_2$ is truncated in grading $|T_2|=|v^{p^{t_2}}M_2|=2(p^{t_2}(i_2+1)+1)$. In Figure \ref{unw}, elements $c$ are in grading $M_2'$, and $c'$ is in grading $M_1'+2(p-1)$. The necessary condition for nontrivial image in $k(1)^*(K_2)$ (and hence $\Delta>1$) is
\begin{equation}\label{TM}|T_2|+2(p-1)\le M_1'+2(p-1)\le M_2'.\end{equation} If this occurs, then we might  have $\Delta$ as large as $\dfrac{M_2'-M_1'}{2(p-1)}+1$.
We now show in Lemma \ref{hell} that if (\ref{TM}) holds, then $(M_2'-M_1')/(2(p-1))<t$, establishing the claim made earlier about $\Delta\le t$.

 We restrict to $p=5$, $i=4\ell$ for simplicity, and so that the reader can refer to Table \ref{tabl} as an aid. The argument easily generalizes to any prime and any congruence. We divide everything by 2 as was done above, and also subtract off the $+1$ which occurs in formulas for $|M|$ and $|T|$, so the numbers will be 1 smaller than those in the table.
\begin{lem}\label{hell} If $t_1>t_2$ and
$$5^{t_2}(4\ell_2+1)+4\le 5^{t_1}(4\ell_1+5)-4r'(t_1-1)+4\le 5^{t_2}(4\ell_2+5)-4r'(t_2-1),$$
then
$$\tfrac14\bigl(5^{t_2}(4\ell_2+5)-4r'(t_2-1)-(5^{t_1}(4\ell_1+5)-4r'(t_1-1))\bigr)<t_1-1.$$\end{lem}
\begin{proof} If there is a counterexample to this, then there is one with $\ell_1=0$, since $\ell_2$ could be decreased by $5^{t_1-t_2}\ell_1$, so it suffices to use $\ell_1=0$. Let $Q(k)=(5^{2k}-1)/24$ (called $q(k)$ in \cite[Lemma 5.3]{DRW}). Then, using \cite[Lemma 5.5]{DRW}, for $t=2k+\delta$ with $\delta=1$ or 2,
$$5^{t+1}-4r'(t-1)=5^{2k+\delta}+16\cdot 5^\delta Q(k)+4k+4\cdot5^{\delta-1}.$$
Since $16\cdot 5^\delta Q(k)+4k+4\cdot5^{\delta-1}<3\cdot 5^{2k+\delta}$, the hypothesis of the lemma says that $5^{t_1+1}-4r'(t_1-1)$ mod $4\cdot 5^{t_2}$ lies in the mod-$(4\cdot 5^{t_2})$ interval $[5^{t_2},5^{t_2+1}-4r'(t_2-1)-4]$.

Let $t_1=2k_1+\delta_1$ and $t_2=2k_2+\delta_2$. The condition is restated as
\begin{equation}\label{x1}5^{2k_1+\delta_1}+16\cdot 5^{\delta_1}Q(k_1)+4k_1+4\cdot 5^{\delta_1-1}\end{equation}
lies in the mod-$(4\cdot5^{t_2})$ interval
\begin{equation}\label{I}[5^{t_2},5^{t_2}+16\cdot5^{\delta_2}Q(k_2)+4k_2+4\cdot5^{\delta_2-1}-4].\end{equation}

Let $\delta_2=1$. The reduction mod $4\cdot 5^{t_2}$  of (\ref{x1}) is
\begin{equation}\label{x2}5^{t_2}+16\cdot 5^{\delta_1}Q(k_2)+4k_1+4\cdot5^{\delta_1-1}.\end{equation}
Let $\delta_1=2$. Then $5^{t_2}+16\cdot5^{\delta_1}Q(k_2)>4\cdot 5^{t_2}$ and equals $5^{2k_2+2}-(2000Q(k_2-1)+100)$, so (\ref{x2}) will  first be in the interval (\ref{I}) when $4k_1+20=2000Q(k_2-1)+100$, hence $k_1=500Q(k_2-1)+20$, so $t_1=1000Q(k_2-1)+42$. The left hand side of the conclusion of the lemma is $\frac18(M_2'-M_1')$ with $M_1'$ and $M_2'$ as in (\ref{TM}). For $k_1=500Q(k_2-1)+20$, the value of $M_1'$ is at the left end of the interval (\ref{I}), and so $\frac18(M_2'-M_1')$ equals $\frac14$ times the length plus 4 of (\ref{I}), which is
$$20Q(k_2)+k_2+1=500Q(k_2-1)+k_2+21=\tfrac12t_1+k_2.$$
Since $k_2<<t_1$, this is less than $t_1-1$, as claimed.
If $k_1$ is increased from the value $500Q(k_2-1)+20$, the value of $t_1$ increases, while $M_2'-M_1'$ decreases, since $M_1'$ is moving through the interval, so the inequality asserted in the lemma is satisfied more strongly.

Now, with $\delta_2=1$ continuing, let $\delta_1=1$. Since $k_1>k_2$, (\ref{x2}) lies outside the interval (\ref{I}) until $80Q(k_2)+4k_1+4=4\cdot 5^{t_2}$, so
$$k_1=5^{2k_2+1}-20Q(k_2)-1=100Q(k_2)+4$$
and $t_1=200Q(k_2)+9$. Again $\frac18(M_2'-M_1')=20Q(k_2)+k_2+1\approx\frac1{10}t_2+k_2$, so the conclusion of the lemma is satisfied more strongly.

A similar analysis works when $\delta_2=2$. In this case $\frac18(M_2'-M_1')\approx\frac12t_1+k_2$ if $\delta_1=1$, and $\frac18(M_2'-M_1')\approx\frac1{10}t_1+k_2$ if $\delta_1=2$.\end{proof}

We close this section by explaining how Theorems \ref{E2} and \ref{diffl} lead to the descriptions of $ku^*(K_2)$ given in Theorems \ref{evthm} and \ref{oddthm}, modulo exotic extensions. We begin with the portion in even gradings and restrict our attention to odd $p$. All elements in the $P[h_0,v,y_1]$ part of Theorem \ref{E2} support differentials of type  (\ref{1}). Note that $y_0^{p^k-1}=y_0^{p-1}y_1^{p^{k-1}-1}=\prod_{j=0}^{k-1}y_j^{p-1}$. The first is easiest to write, the second occurs in Theorem \ref{E2}, and the third in \ref{ABdef} and Figure \ref{oddchart}. From \ref{ABdef}, $y_0^{p^k-1}z_0$ is in $A_k$ for $k\ge1$, the bottom right element in Figure \ref{oddchart}. Then
\begin{equation}\label{Pyy}P[y_1]y_0^{p-1}z_0=\bigoplus \M_k^A\cdot y_0^{p^k-1}z_0\subset \bigoplus \M_k^AA_k.\end{equation}
The first part occurs in Theorem \ref{E2} and the last part in Theorem \ref{evthm}.

Now we consider $P[y_1]\ot\bigoplus_{j\ge1} W_j\ot TP_{p-1}[z_j]\ot\L_{j+1}$ in Theorem \ref{E2}. The $\bigoplus$ part is all monomials $z_\ell M$ with $\ell\ge1$ and $M\in\L_\ell$. From Theorem \ref{diffl}, $y_1^iz_\ell M$ supports a differential (\ref{2}) if  $\ell\ge\nu(i)+2$, while those with $\nu(i)\ge\ell-1$ are hit by differentials (\ref{3}) and (\ref{4}), yielding $v$-towers with heights as given in \ref{ABdef}. These are all monomials in $\bigoplus_{\ell\ge1}P[y_\ell,y_{\ell+1},\ldots]z_\ell\L_\ell$. From \ref{ABdef} or (\ref{Bk}), the generators of the $v$-towers in $B_k$ are all
$$z_j\prod_{i=j}^{k-1}\{z_i^{p-1},y_i^{p-1}\},\ 1\le j\le k.$$
Let $(z_\ell M)_i$ be the $y_i^ez_i^{e'}$ factors of $ M$. Then $\M_kB_k$ consists of all monomials $z_\ell M$ such that $(z_\ell M)_i$ equals $y_i^{p-1}$ or $z_i^{p-1}$ for $\ell\le i<k$, but not for $ i=k$, and so every monomial $z_\ell M$ is in a unique $\M_kB_k$. From Theorem \ref{diffl}, $z_\ell M$ has $v$-height $p^\ell$ if and only if $M$ contains no $z$-factors, which explains the split into $\M_k^A$ and $\M_k^B$ in Theorem \ref{evthm}.

Now we address the odd gradings. The $P[h_0,v,y_1]vq$ part of Theorem \ref{E2} is totally removed either as sources (\ref{3}) or targets (\ref{1}) of differentials. See grading 17 in Figure \ref{low} for a nice illustration.
The $qy_1^{i-1}S_{\nu(i)+1,\ell}$ part of Theorem \ref{oddthm} is formed from $TP_{\nu(i)+2}[v]qy_1^{i-1}W_\ell$ in \ref{E2} using (\ref{2}). The generators of $S_{\nu(i)+1,\ell}$ are $z_{1,\ell},\ldots,z_{\ell-\nu(i)-1,\ell}$, but to see the differential from (\ref{2}), one should write $z_{t,\ell}=z_{t,t+\nu(i)+1}Z_{t+\nu(i)+1}^\ell$, where
\begin{equation}\label{Zdef}Z_i^j=(z_i\cdots z_{j-1})^{p-1}\text{ for }j>i, \text{ with }Z_i^i=1.\end{equation}
The remaining generators of $qy_1^{i-1}W_\ell$, namely $qy_1^{i-1}z_{j,\ell}$ with $ \ell-\nu(i)\le j\le\ell$, support differentials (\ref{4}). There can be no unexpected exotic extensions among these
summands for the reason noted at the end of Section \ref{intro}. The $\ker(p)$ elements in the $S$ summands play a very important role in the exact sequence.

To see that there are no unexpected exotic extensions between these summands

\section{Exotic extensions}\label{extnsec}
In this section, we prove the following expansion of (\ref{extns}).
\begin{thm}\label{extnthm} If  $i\ge0$ and $k\ge k_0$,
$$py_{k}^iy_{k-1}^{p-1}z_{k-1}=v^{p^{k-1}(p-1)}y_{k}^iz_{k}$$
with an additional term $vy_{k}^iy_{k-1}^{p-1}z_{k-2}^p$ if $k\ge k_0+2$.\end{thm}

The additional term is seen in Ext, and will be ignored in the rest of this section. We have included the factor $y_{k}^i$, which is not automatic since $y_{k}^i$ is not a permanent cycle. Since, for example, $y_{k+1}=y_k^p$, we need not consider $y_i$ for $i>k$. It is automatic that this formula can be multiplied by $z_j$'s, since they do survive the spectral sequence.

The extension is deduced from the exact sequence
$$ku^*(K_2)\mapright{\cdot p} ku^*(K_2)\longrightarrow k(1)^*(K_2)$$
and the fact that $v^{r'(k-1)}y_{k}^iz_{k}=0$ in $k(1)^*(K_2)$ with $r'(k-1)\ge p^k(p-1)$. Thus $v^{r'(k-1)}y_{k}^iz_{k}$ must be divisible by $p$ in $ku^*(K_2)$, and, as we will show, the $v$-tower on $y_{k}^iy_{k-1}^{p-1}z_{k-1}$ provides the only classes that can do the dividing. Once we know the division formula toward the end of the $v$-tower, we can deduce that it holds earlier in the tower, as well. For example, $r'(2)=p^3-p^2+p-2$, which is  the height in the top $v$-tower in Figure \ref{oddchart} where the extensions into it do not also involve an $h_0$-extension. We deduce the extensions from the earlier part of the $v$-tower on $y_2^{p-1}z_2$ by naturality.

We illustrate in Figure \ref{extnfig}, using the notation of the preceding section. Thus $T_i$ is the class satisfying $d_r(T_i)=v^rM_i$, Here the portion of the top tower to the right of $M_1'$ must be divisible by $p$. The tower providing the extension must have $M_1'\le |M_2|<|M_1|$ and $|T_2|\le |T_1|$.

\bigskip

\begin{minipage}{6in}
\begin{fig}\label{extnfig}

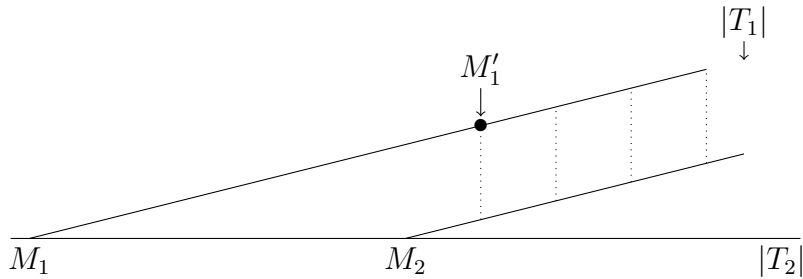
{\bf Conditions for extension}

\begin{center}

\begin{\tz}[scale=.25]
\draw (0,0) -- (36,9);
\draw (20,0) -- (38,4.5);
\draw (-1,0) -- (41,0);
\node at (24,6) {$\bullet$};
\node at (0,-1.2) {$M_1$};
\node at (20,-1.2) {$M_2$};
\draw [dotted] (24,1) -- (24,6);
\draw [dotted] (28,2) -- (28,7);
\draw [dotted] (32,3) -- (32,8);
\draw [dotted] (36,4) -- (36,9);
\node at (24,9) {$M_1'$};
\draw [->] (24,8) -- (24,6.5);
\draw [->] (38,10.5) -- (38,9.5);
\node at (38,11.5) {$|T_1|$};
\node at (40,-1.2) {$|T_2|$};
\end{\tz}\end{center}
\end{fig}
\end{minipage}

\bigskip
As we did for the differentials in the previous section, we will perform the argument for $p=5$. It will be clear that it generalizes to an arbitrary odd prime, and with minor modification to $p=2$. Also, we use $i=4\ell$ in Theorem \ref{extnthm}. If instead we used $i=4\ell+d$ for $1\le d\le 3$, it will just add the same amount to the quantities $|M|$, $|T|$, and $M'$ involved in the argument. We can use Table \ref{tabl} to envision the analysis, with the $t$ there replaced by $k$. For a monomial $M(\ell,k)=y_k^{4\ell}z_k$, we have, after dividing by 2, $|M|=5^k(4\ell+5)+1$, $|T|=5^k(4\ell+1)+1$, and $5^k(4\ell+1.16)+1<  M'\le 5^k(4\ell+1.8)+1$, using (\ref{r5}).
With $M_1$ and $M_2$ as in Figure \ref{extnfig}, we will show that $M_2(5\ell+1,k-1)$ is the unique monomial satisfying the  inequalities stated just before Figure \ref{extnfig} for $M_1(\ell,k)$.
Note that $M(5\ell+1,k-1)=y_k^{4\ell}y_{k-1}^4z_{k-1}$. We omit the $+1$ in all the formulas.

The inequalities are satisfied by $M_2(5\ell+1,k-1)$ since
$$5^k(4\ell+1.8)\le 5^{k-1}(4(5\ell+1)+5)<5^k(4\ell+5)\text{ and }5^{k-1}(4(5\ell+1)+1)\le 5^k(5\ell+1).$$
If $k_2\ge k$, then the first inequality, after dividing by $5^k$, becomes
$$4\ell+1.8\le 5^{k_2-k}(4\ell_2+5)<4\ell+5,$$
which cannot be satisfied since the middle term is $\equiv1$ mod 4. If $k_2<k-1$, then
$$M_1'-|T_1|>5^k\cdot.16\ge4\cdot5^{k_2}=|M_2|-|T_2|,$$
which is inconsistent with two of the inequalities. Let $k_2=k-1$. If $\ell_2<5\ell+1$, then
$$|M_2|=5^{k-1}(4\ell_2+5)\le5^{k-1}(4\cdot5\ell+5)<5^k(4\ell+1.16)<M_1',$$
contradicting one of the inequalities. If $k_2=k-1$ and $\ell_2>5\ell+1$, then
$$|T_2|\ge 5^{k-1}(4(5\ell+2)+1)>5^k(4\ell+1)=|T_1|,$$
contradicting one of the inequalities.

We deduce that $M_2=y_k^{4\ell}y_{k-1}^4z_{k-1}$, as claimed. We should perhaps have noted that the extensions could not have come from classes with more than one $z_j$-factor, because these are $z_j$ times a class on which the extensions have already been determined.

\section{Proposed formulas for the exact sequence (\ref{LES})}\label{LESsec}
In this section we propose what we feel must be the correct complete formulas for the exact sequence (\ref{LES}). Some homomorphisms are forced by naturality, but many others involve significant filtration jumps. However, they all occur in several families with nice properties. The 10-term exact sequence (\ref{10}) shows how the $S_{k,\ell}$ portions and the exotic extensions yield compatibility of the differing $v$-tower heights in $ku^*(K_2)$ and $k(1)^*(K_2)$. In Section \ref{allsec}, we show that all elements of $k(1)^*(K_2)$ are accounted for exactly once in these homomorphisms, which implies that there can be no more exotic extensions. This does not require us to prove that our homomorphism formulas are actually correct, as discussed at the end of Section \ref{intro}. We will focus on the case when $p$ is odd. We could incorporate all primes together at the expense of involving the parameter $k_0$, but things are complicated enough without that. In an earlier version of this paper (\cite{2}), a thorough analysis when $p=2$ was performed.

We propose that (\ref{LES}) can be split into exact sequences of length 4 and 10 (not including 0's at the end). There are subgroups of $k(1)^*(K_2)$ called $G_k^1$ and $G_k^2$ for $k\ge1$ and $G^i_{k,\ell}$ for $3\le i\le6$ and $1\le k<\ell$ such that there are exact sequences
\begin{equation}\label{Aseq}0\to G_k^1\to A_k\mapright{p}A_k\to G_k^2\to 0\end{equation}
for $k\ge1$, and, for $1\le k<\ell$,
\begin{eqnarray}\nonumber0&\to& G^3_{k,\ell}\to y_kB_{k}Z_k^\ell\mapright{p} y_kB_{k}Z_k^\ell\to G^4_{k,\ell}\to y_1^{p^{k-1}-1}q S_{k,\ell}\\
&\mapright{p}&y_1^{p^{k-1}-1}q S_{k,\ell}\to G^5_{k,\ell}\to B_{k}z_\ell\mapright{p}B_{k}z_\ell\to G^6_{k,\ell}\to0,\label{10}\end{eqnarray}
with $Z_k^\ell$ as defined in (\ref{Zdef}).
The  sequence (\ref{Aseq}) can be tensored with $TP_{p-1}[y_k]\ot P[y_{k+1}]$, while (\ref{10}) can be tensored with $TP_{p-1}[y_k]\ot P[y_{k+1}]\ot TP_{p-1}[z_\ell]\ot \L_{\ell+1}$.
If $p$ is odd, there are also exact sequences
\begin{equation}\label{78}0\to G^7_{k,e}\to B_kz_k^e\mapright{p} B_kz_k^e\to G^8_{k,e}\to 0\end{equation}
for $k\ge1$ and $1\le e\le p-2$. This can be tensored with $P[y_k]\ot\L_{k+1}$.

One can verify that the totality of $A_k$ and $B_k$ groups in these exact sequences agrees with that in Theorem \ref{evthm}.
We will study these exact sequences by breaking them up into short exact sequences and isomorphisms involving kernels and cokernels of $\cdot p$.

Let $K_k^A=\ker(\cdot p|A_k)$, $K_k^B=\ker(\cdot p|B_k)$, $C_k^A=\coker(\cdot p|A_k)$, and $C_k^B=\coker(\cdot p|B_k)$. There are important elements $g_k\in K_k^A$ and $K_k^B$ defined (up to unit coefficients) by $g_1=z_1$, $g_2=v^{p-2}z_2$, and, for $k\ge1$,
\begin{equation}\label{gdef}g_{k+2}=v^{r'(k)-1}z_{k+2}+g_ky_k^{p-1}z_{k+1}^{p-1}.\end{equation}
To see that this is in $\ker(\cdot p)$, we use (\ref{extns}) to see that $p\cdot v^{r'(k)-1}z_{k+2}=v^{r'(k)}z_{k+1}^p$, and that the $v^{r'(k-2)-1}z_k$ term in $g_k$ yields $v^{r'(k-2)-1}v^{p^k(p-1)}z_{k+1}z_{k+1}^{p-1}$ in $p\cdot g_ky_k^{p-1}z_{k+1}^{p-1}$. Using  (\ref{rprec}), these terms cancel. Other terms in $p\cdot g_ky_k^{p-1}z_{k+1}^{p-1}$ yield 0 since $g_k\in\ker(\cdot p)$.

The $v$-towers in $K_k^A$ are generated by \begin{equation}\label{gns}g_k\text{ and }g_jz_j^{p-1}\prod_{i=j+1}^{k-1}\{z_i^{p-1},y_i^{p-1}\},\ 1\le j\le k-1.\end{equation} For example, using Figure \ref{oddchart} when $k=3$, these are
 $g_3=v^{p^2-p-1}z_3+y_1^{p-1}z_1z_2^{p-1}$, $g_2z_2^{p-1}=v^{p-2}z_2^p$, $g_1z_1^{p-1}z_2^{p-1}$, and $g_1z_1^{p-1}y_2^{p-1}$.
 The $v$-heights are $p^k-(r'(k-2)-1)$ for $g_k$, and $p^j-j-(r'(j-2)-1)$ for the others, since they are determined by $v$-heights of $z_j$ in $B_k$. The map $G_k^1\to K_k^A$ sends $w_k$ to $g_k$ and
 \begin{equation}w_jP\mapsto g_jP\text{ for }P=z_j^{p-1}\prod_{i=j+1}^{k-1}\{z_i^{p-1},y_i^{p-1}\}\label{wg},\end{equation}
 with $w_j$ as in \ref{rprop} and \ref{DRWthm}. The $v$-height of $w_j$ is $r(j)$ if it is not accompanied by $z_j$, and $r'(j-1)$ if it is. By (\ref{r3}) and ((\ref{r2}) and (\ref{r1})) the $v$-heights agree, so (\ref{wg}) is an isomorphism on $v$-towers.

 For $L=K_k^A$ or $K_k^B$ or $C_k^A$ or $C_k^B$, we say that a $\zp$ in $L$ is a class of $v$-height 1 in $L$ which is not part of a larger $v$-tower in $L$. There is one $\zp$ in $K_3^A$, as can be seen in Figure \ref{oddchart}. This is the element $v^{p-2}y_1^{p-1}z_1z_2^{p-1}$. Note that for $i<p-1$, $v^iy_1^{p-1}z_1z_2^{p-1}+v^{i+p^2-p-1}z_3$ is part of a $v$-tower in $K_3^A$, which continues with the elements $v^{i}z_3$ for  $i>p^2-3$, but
 $v^iy_1^{p-1}z_1z_2^{p-1}$ itself is in $K_3^A$ only for $i=p-2$. Using \ref{ABdef}, we find that the $\zp$'s in $K_k^A$ are
 \begin{equation}\label{KkA}v^{p^t-t-1}(y_t\cdots y_{j-1})^{p-1}z_tz_j^{p-1}\prod_{i=j+1}^{k-1}\{z_i^{p-1},y_i^{p-1}\}\text{ for }1\le t<j<k.\end{equation}
 For example, the elements $v^{p-2}(y_1y_2)^{p-1}z_1$ and $v^{p^2-3}y_2^{p-1}z_2$ in Figure \ref{oddchart} yield elements in $K_4^A$ after being multiplied by $z_3^{p-1}$. The basic formula for the homomorphism from part of $k(1)^*(K_2)$ to $\zp$'s in various $K_k^A$ and $K_k^B$, possibly tensored with other classes as in Theorem \ref{evthm}, is
 \begin{equation}\label{KK}\bigl(q(y_1\cdots y_t)^{p-1}z_{j-t,j}\mapsto v^{p^t-t-1}y_t^{p-1}z_tz_j\bigr)\ot P[y_j]\ot TP_{p-1}[z_j]\ot \L_{j+1}\text{ for }j>t\ge1.\end{equation}
 The domain elements are in the second half of the third line of Theorem \ref{DRWthm}. The ones that are in $G^1_k$ in the isomorphism $G^1_k\to K_k^A$ can be extracted using (\ref{KkA}).

 The isomorphism $G^3_{k,\ell}\to y_kK_k^B Z_k^\ell$ in (\ref{10}) is given using formulas analogous to (\ref{wg}) and (\ref{KK}). There are several minor differences. One is that the $v$-tower on $y_kg_kZ_k^\ell$ is truncated due to $v^{p^k-k}z_k=0$ in $B_k$ (as opposed to $v^{p^k}z_k=0$ in $A_k$). This is compatible with the fact that the $v$-height of $w_kz_k$ in $k(1)^*(K_2)$ is $k$ less than that of $w_k$, using Theorem \ref{DRWthm} and (\ref{r1}). The other is that $K_k^B$ has additional $\zp$'s \begin{equation}\label{KkB}v^{p^t-t-1}(y_t\cdots y_{k-1})^{p-1}z_t\text{ for }1\le t\le k-1,\end{equation} as seen in Figure  \ref{oddchart} when $k=3$, but these are always multiplied by higher $z$'s, and so (\ref{KK}) applies.

The isomorphisms $C_k^A\to G_k^2$ and $C_k^Bz_\ell\to G_{k,\ell}^6$ are defined simply by sending an element to one with the same name. Moreover $C_k^A=C_k^B$ except for $(y_0\cdots y_{k-1})^{p-1}z_0\in C_k^A-C_k^B$. When $k=3$, we see that the $\zp$'s in $C_k^B$ are $\{z_1^pz_2^{p-1},\,z_2^p,\,y_2^{p-1}z_1^p\}$ in Figure \ref{oddchart}.\footnote{The class $y_2^{p-1}z_1^{p-1}$ should really be called $y_2^{p-1}z_1^{p-1}+v^{p^2(p-1)-1}z_3$ so that $v$ times it is divisible by $p$, hence 0 in $C_k^B$, but we will ignore this fine-tuning.} For future reference,
\begin{equation}\label{Ck}\text{$\zp$'s in }C_k^B\text{ are} \bigl\{z_t^p\prod_{i=t+1}^{k-1}\{z_i^{p-1},y_i^{p-1}\bigr\}:\, 1\le t<k\}.\end{equation}
The corresponding elements in $k(1)^*(K_2)$ are from the third line of \ref{DRWthm}.

The $v$-towers in $C_k^A=C_k^B$ are generated by
\begin{equation}\label{Ctow}z_k\text{ and }y_t^{p-1}z_t\prod_{i=t+1}^{k-1}\{z_i^{p-1},y_i^{p-1}\},\ 1\le t<k.\end{equation}
We will show that the $v$-height of $z_k$ in $C_k^B$ is $r'(k-1)$, which equals its $v$-height in $k(1)^*(K_2)$. It follows from \ref{ABdef} that the $v$-height of $y_t^{p-1}z_t\prod_{i=t+1}^{k-1}\{z_i^{p-1},y_i^{p-1}\}$ equals $r'(t-1)$, establishing the isomorphisms out of $C_k^A$ and $C_k^Bz_\ell$. In Figure \ref{oddchart}, the $v$-height of $z_3$ equalling $p^3-p^2+p-2=r'(2)$ is apparent.

The proof of the claim about $v$-heights is by induction. By (\ref{rprec}), $r'(k-1)-r'(k-3)=p^{k-1}(p-1)-1$. Let $D=(|z_k|-|y_{k-1}^{p-1}z_{k-1}|)/(2(p-1))=p^{k-1}(p-1)$. This is the filtration on the $z_k$-tower above the element $y_{k-1}^{p-1}z_{k-1}$. We show that $v^{i-1+D}z_k$ is divisible by $p$ if and only if $v^iz_{k-2}$ is divisible by $p$. Thus the difference of the $v$-heights in cokernels equals the difference of the corresponding $r'$ values. From Theorem \ref{extnthm}, we have
$$pv^{i-1}y_{k-1}^{p-1}z_{k-1}=v^{i-1+D}z_k+v^iy_{k-1}^{p-1}z_{k-2}^p.$$
The claim follows, since $v^iy_{k-1}^{p-1}z_{k-2}^p$ is divisible by $p$ if and only if $v^iz_{k-2}$ is, by \ref{ABdef}.

The analysis of (\ref{78}) is extremely similar.

Now $S_{k,\ell}$ becomes involved. Let $S_{k,\ell}^K=\ker(\cdot p|S_{k,\ell})$ and $S_{k,\ell}^C=\coker(\cdot p|S_{k,\ell})$. Then $S_{k,\ell}^K$ consists of $TP_{k+1}[v]\langle z_{1,\ell}\rangle$ plus $\zp$'s on $v^kz_{i,\ell}$ for $2\le i\le \ell-k$, while $S_{k,\ell}^C$ has $TP_{k+1}[v]\langle z_{\ell-k,\ell}\rangle$ plus $\zp$'s on $z_{i,\ell}$ for $1\le i<\ell-k$. Next we consider the short exact sequence
\begin{equation}\label{G4seq}0\to y_kC_k^BZ_k^\ell\mapright{\phi} G_{k,\ell}^4\mapright{\psi} y_1^{p^{k-1}-1}qS_{k,\ell}^K\to 0.\end{equation}
The map $\phi$ sends everything except the $v$-tower on $y_kz_kZ_k^\ell$ to classes with the same name, and the heights of these $v$-towers agree, as seen above. The class $y_kz_kZ_k^\ell=y_kz_{k,\ell}$ maps to a $\zp$ with the same name in $k(1)^*(K_2)$. We have $\psi(w_kw_{k+1}Z_{k+1}^\ell)=qy_1^{p^{k-1}-1}z_{1,\ell}$. Then $v^{k+1}w_kw_{k+1}Z_{k+1}^\ell\in\ker(\psi)$, and we have
$$\phi(vy_kz_{k,\ell})=v^{k+1}w_kw_{k+1}Z_{k+1}^\ell.$$
We illustrate this in the schematic Figure \ref{fig7}, in which $X$, $\circ$, and $\bullet$  map to elements with the same symbol. The expressions at the end of the $v$-towers are their $v$-heights. In particular, $v^{r'(k-1)}y_kz_{k,\ell}=0$ in $y_kC_k^BZ_k^\ell$. The $v$-heights agree by (\ref{r1}), and the gradings match by an induction proof.
The $\zp$'s in $y_1^{p^{k-1}-1}qS^K_{k,\ell}$ are hit by $\psi(y_kz_{i+k-1,\ell})=y_1^{p^{k-1}-1}qv^kz_{i,\ell}$, $2\le i\le\ell-k$, another interesting filtration jump.

\bigskip
\begin{minipage}{6in}
\begin{fig}\label{fig7}

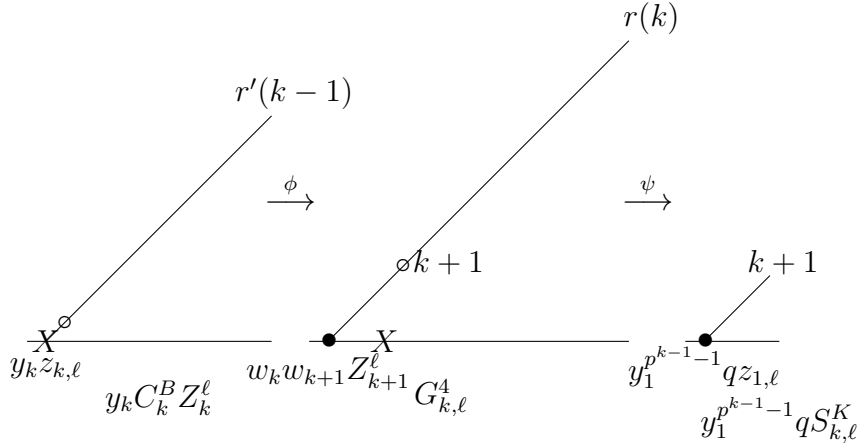
{\bf Towers in exact sequence.}

\begin{center}

\begin{\tz}[scale=.25]

\draw (0,0) -- (13,0);
\draw (1,0) -- (13,12);
\draw (15,0) -- (32,0);
\draw (16,0) -- (32,16);
\draw (35,0) -- (40,0);
\draw (36,0) -- (39.5,3.5);
\node at (2,1) {$\circ$};
\node at (14.2,13.2) {$r'(k-1)$};
\node at (20,4) {$\circ$};
\node at (16,0) {\lb};
\node at (16,-1.6) {$w_kw_{k+1}Z_{k+1}^\ell$};
\node at (36,0) {\lb};
\node at (36,-1.4) {$y_1^{p^{k-1}-1}qz_{1,\ell}$};
\node at (40.2,4.2) {$k+1$};
\node at (22.4,4.2) {$k+1$};
\node at (33.2,17.2) {$r(k)$};
\node at (1,0) {$X$};
\node at (19,0) {$X$};
\node at (7,-3) {$y_kC_k^BZ_k^\ell$};
\node at (22,-3) {$G_{k,\ell}^4$};
\node at (40,-4.2) {$y_1^{p^{k-1}-1}qS_{k,\ell}^K$};
\node at ((14,7) {$\mapright{\phi}$};
\node at  (33,7) {$\mapright{\psi}$};
\node at (1,-1.5) {$y_kz_{k,\ell}$};
\end{\tz}
\end{center}
\end{fig}
\end{minipage}

\bigskip
Finally we consider the short exact sequence
\begin{equation}\label{G5seq}0\to y_1^{p^{k-1}-1}qS^C_{k,\ell}\mapright{\phi'}G^5_{k,\ell}\mapright{\psi'}K_k^Bz_\ell\to0.\end{equation}
Similarly to (\ref{gns}), the generators of $v$-towers in $K_k^B$ are $g_k$ and, for $1\le j<k$, elements of the form $g_jz_j^{p-1}\prod_{j+1}^{k-1}\{z_i^{p-1},y_i^{p-1}\}$. The morphism $\psi'$ is determined  by $w_j\mapsto g_j$. The $v$-heights of the corresponding elements in $k(1)^*(K_2)$ and $K_k^B$ both equal $r'(j-1)$ for $j<k$. However, the $v$-height of $w_kz_\ell$ is $r(k)$, which is $k$ greater than $r'(k-1)$. We have $\phi'(vy_1^{p^{k-1}-1}qz_{\ell-k,\ell})=v^{r'(k-1)}w_kz_\ell$. The class $y_1^{p^{k-1}-1}qz_{\ell-k,\ell}$ at the base of the $v$-tower maps to a $\zp$ with the same name. The picture is quite similar to Figure \ref{fig7} with $k+1$ and $r'(k-1)$ interchanged.

The $\zp$ classes $y_1^{p^{k-1}-1}qz_{i,\ell}$ for $1\le i<\ell-k$ are mapped by $\phi'$ to classes with the same name in $G_{k,\ell}^5\subset k(1)^*(K_2)$. The $\zp$'s in $K_k^Bz_\ell$ are of the same form as in (\ref{KkA}), and are hit by analogues of (\ref{KK}).

\section{All accounted for}\label{allsec}
In this section, we show that all elements of $k(1)^*(K_2)$ are involved in exactly one of the homomorphisms involving some $G$-group described in the preceding section. As discussed earlier, this implies that there can be no exotic extensions in $ku^*(K_2)$ other than those in (\ref{extns}), because an additional extension would decrease the number of elements in $\ker(\cdot p|ku^*(K_2))$ and $\coker(\cdot p|ku^*(K_2))$, and these must correspond to elements of $k(1)^*(K_2)$. It also provides an excellent check on our analysis.

Let $p$ be odd, and
$$G^i=\begin{cases}\ds\bigoplus_{k\ge1}G_k^i\ot TP_{p-1}[y_k]\ot P[y_{k+1}]&1\le i\le 2\\
\ds\bigoplus_{1\le k<\ell}G_{k,\ell}^i\ot TP_{p-1}[y_k]\ot P[y_{k+1}]\ot TP_{p-1}[z_\ell]\ot\L_{\ell+1}&3\le i\le 6\\
\ds\bigoplus_{k\ge1}\bigoplus_{e=1}^{p-2}G_{k,e}^i\ot P[y_k]\ot\L_{k+1}&7\le i\le8.\end{cases}$$

\begin{thm}\label{allthm} $G^1\oplus\cdots\oplus G^8$ equals $k(1)^*(K_2)$, as described in Theorem \ref{DRWthm}.\end{thm}
As throughout the paper, $\zp$'s coming from $E_1$-free submodules of $H^*(K_2)$ are ignored here. The remainder of this section is devoted to the proof of Theorem \ref{allthm}. There are four parts of Theorem \ref{DRWthm}. We deal with them one-at-a-time.

{\bf Case} 1. $P[y_1]y_0^{p-1}z_0$. In (\ref{Pyy}), it is shown that these classes form a subset of $\bigoplus \M_k^AA_k$, and they map to classes with the same name in $G^2$.

{\bf Case} 2. $\bigoplus_{j>0}TP_{r(j)}[v]\ot P[y_{j+1}]\ot TP_{p-1}[y_j]\ot \Ebar[w_{j}]\ot E[w_{j+1}]\ot \L_{j+1}$. The generators of $v$-towers of height $r(j)$ occur in $G^1$, $G^4$, and $G^5$. From (\ref{wg}), only $w_j$ is in $G_j^1$. So $G^1$ has $TP_{p-1}[y_j]\ot P[y_{j+1}]w_j$. From Figure \ref{fig7}, $G^4_{j,\ell}$ has $w_jw_{j+1}Z_{j+1}^\ell$. Note that $\bigoplus_\ell Z_{j+1}^\ell TP_{p-1}[z_\ell]\ot \L_{\ell+1}=\L_{j+1}$, since the $\ell$-component gives the monomials whose smallest non-$(p-1)$-power is a power of $z_\ell$, so $G^4$ contains $P[y_{j+1}]\ot TP_{p-1}[y_j]w_{j}w_{j+1}\ot \L_{j+1}$. From the analysis following (\ref{G5seq}), $G^5_{j,\ell}$ has only $w_jz_\ell$ of $v$-height $r(j)$, so $G^5$ will have $P[y_{j+1}]\ot TP_{p-1}[y_j]w_{j}\ot \Lbar_{j+1}$. Thus $G^1\oplus G^5$ contains the part without $w_{j+1}$, while $G^4$ contains the part with $w_{j+1}$.

{\bf Case} 3. $\bigoplus_{j\ge1}TP_{r'(j-1)}[v]\ot P[y_{j}]\ot E[w_{j}]\ot\overline{TP}_p[z_{j}]\ot \L_{j+1}$. The generators of $v$-towers of height $r'(j-1)$ occur in each $G^i$ as follows. \begin{itemize}
\item [$G^1$:] $\dstyle{w_jz_j^{p-1}\bigoplus_{k\ge j+1}TP_{p-1}[y_k]\ot P[y_{k+1}]\ot\bigoplus_{i=j+1}^{k-1}\{z_i^{p-1},y_i^{p-1}\}}$. This can be deduced from (\ref{wg}).
\item[$G^2$:] From (\ref{Ctow}),
$$z_jTP_{p-1}[y_j]\ot P[y_{j+1}]\oplus y_j^{p-1}z_j\bigoplus_{k\ge j+1}TP_{p-1}[y_k]\ot P[y_{k+1}]\ot\prod_{i=j+1}^{k-1}\{z_i^{p-1},y_i^{p-1}\}.$$
\item[$G^3$:] We use (\ref{gns}) and (\ref{wg}) and adapt some arguments used in Case 2 to obtain
$$w_jz_j^{p-1}\biggl(\overline{TP}_p[y_j]\ot P[y_{j+1}]\ot\L_{j+1}\oplus\bigoplus_{k\ge j+1}\overline{TP}_p[y_k]P[y_{k+1}]z_k^{p-1}\L_{k+1}\prod_{i=j+1}^{k-1}\{z_i^{p-1},y_i^{p-1}\}\biggr).$$
\item[$G^4$:] We use (\ref{Ctow}) and (\ref{G4seq}) to obtain
$$y_j^{p-1}z_j\bigoplus_{k\ge j+1}\overline{TP}_p[y_k]\ot P[y_{k+1}]z_k^{p-1}\L_{k+1}\prod_{i=j+1}^{k-1}\{z_i^{p-1},y_i^{p-1}\}.$$
\item[$G^5$:] We use (\ref{G5seq}) and $\bigoplus_{\ell>k}z_\ell TP_{p-1}[z_\ell]\ot\L_{\ell+1}\approx\Lbar_{k+1}$ to obtain
$$w_jz_j^{p-1}\bigoplus_{k\ge j+1}TP_{p-1}[y_k]\ot P[y_{k+1}]\ot\Lbar_{k+1}\ot\prod_{i=j+1}^{k-1}\{z_i^{p-1},y_i^{p-1}\}.$$
\item[$G^6$:] We combine the analysis for $G^2$ and the observation used for $G^5$ to obtain
\begin{eqnarray*}&&z_jTP_{p-1}[y_j]\ot P[y_{j+1}]\ot\Lbar_{j+1}\\
&\oplus& y_j^{p-1}z_j\bigoplus_{k\ge j+1}TP_{p-1}[y_k]\ot P[y_{k+1}]\ot\Lbar_{k+1}\ot\prod_{i=j+1}^{k-1}\{z_i^{p-1},y_i^{p-1}\}\end{eqnarray*}
\item[$G^7$:] Similarly to $G^3$, we have
$$\bigoplus_{e=1}^{p-2}\biggl(w_jz_j^e\ot P[y_j]\ot\L_{j+1}\oplus w_jz_j^{p-1}\bigoplus_{k\ge j+1}z_k^e\ot P[y_k]\ot\L_{k+1}\ot\prod_{i=j+1}^{k-1}\{z_i^{p-1},y_i^{p-1}\}\biggr).$$
\item[$G^8$:] Using (\ref{Ctow}), we get
$$\bigoplus_{e=1}^{p-2}\biggl(z_j^{e}\ot P[y_j]\ot\L_{j+1}\oplus y_j^{p-1}z_j\bigoplus_{k\ge j+1}z_k^e\ot P[y_k]\ot\L_{k+1}\ot\prod_{i=j+1}^{k-1}\{z_i^{p-1},y_i^{p-1}\}\biggr).$$
\end{itemize}

We begin by analyzing the portion including the factor $w_j$. We will show that $$G^1\oplus G^3\oplus G^5\oplus G^7= P[y_j]w_j\ot\overline{TP}_p[z_j]\ot\L_{j+1}.$$
Here, and in the remainder of our analysis of Case 3, $G^i$ refers just to the relevant portion of $G^i$, here the part with $TP_{r'(j-1)}[v]w_j$.
 The first part of $G^7$ gives all terms with $z_j^e$ for $1\le e\le p-2$. The remaining part has factors $w_jz_j^{p-1}$, which we will omit writing. Combining $G^1$ and $G^5$ removes the bar in $G^5$. The first part of $G^3$ gives the part with positive exponent of $y_j$, which we now omit.

Let $E_\ell=P[y_\ell]\ot\L_\ell$, thought of as monomials in $y_i$ and $z_i$ for $i\ge\ell$ with exponents $\le p-1$. The remaining parts of the $G^i$'s under consideration combine to
\begin{equation}\label{big}\bigoplus_{k\ge j+1}\biggl(TP_{p-1}[y_k]\oplus y_kz_k^{p-1}TP_{p-1}[y_k]\oplus\bigoplus_{e=1}^{p-2}z_k^eTP_p[y_k]\biggr)\ot E_{k+1}\ot\prod_{i=j+1}^{k-1}\{z_i^{p-1},y_i^{p-1}\}.\end{equation}
We wish to show this equals $E_{j+1}$. The portion in parentheses is all monomials in $TP_p[y_k,z_k]$ except $y_k^{p-1}$ and $z_k^{p-1}$. For a monomial $M$ in $E_{j+1}$, let $M_i$ denote its $y_i^sz_i^t$ factor. The $k$-summand in (\ref{big}) is all monomials $M$ in $E_{j+1}$ for which $k$ is the smallest $i$ such that $M_i$ is neither $y_i^{p-1}$ nor $z_i^{p-1}$. Thus the sum over all $k$ yields all of $E_{j+1}$, as claimed.

A very similar argument shows that the $G^2\oplus G^4\oplus G^6\oplus G^8$ part for Case 3 equals the portion which includes just the 1 in $E[w_j]$; i.e., $P[y_j]\ot \overline{TP}_p[z_j]\ot\L_{j+1}$.

\medskip
{\bf Case} 4. $\bigoplus_{j\ge1} P[y_1]\ot E[q ]\ot \Ebar[z_j^p]\ot \L_{j+1}$. We first consider the part without the $q$, and fix $j$ and omit writing the $z_j^p$. The desired answer is $P[y_1]\ot\L_{j+1}$. These come from the $\zp$'s in $G^2\oplus G^4\oplus G^6\oplus G^8$. Similarly to Case 3, $G^2$ and $G^6$ combine to give
$$\bigoplus_{k\ge j+1}TP_{p-1}[y_k]\ot P[y_{k+1}]\ot\L_{k+1}\ot\prod_{i=k+1}^{j-1}\{z_i^{p-1},y_i^{p-1}\}.$$
This, together with the portion of $G^4$ from $\im(\phi)$ in (\ref{G4seq}) obtained using (\ref{Ck}), and the $\zp$'s in $G^8$ obtained using (\ref{Ck}) give exactly (\ref{big}), which we showed equals $P[y_{j+1}]\ot\L_{j+1}$.\footnote{Here the classes in (\ref{big}) are $\zp$'s and are multiplied by $z_j^p$, whereas in Case 3 they were multiplied by $w_jz_j^{p-1}$ and were generators of $v$-towers of height $r'(j-1)$.} The element $X$ in Figure \ref{fig7} with $k$ replaced by $j$ yields, from $G^4$,
\begin{eqnarray*}&&y_jTP_{p-1}[y_j]\ot P[y_{j+1}]\ot\bigoplus_{\ell>j}Z_{j+1}^\ell TP_{p-1}[z_\ell]\ot\L_{\ell+1}\\
&=&y_jTP_{p-1}[y_j]\ot P[y_{j+1}]\ot\L_{j+1},\end{eqnarray*}
which combines with the portion just obtained to yield $P[y_j]\ot\L_{j+1}$.

The last line of the $G^4_{k,\ell}$ discussion in Section \ref{LESsec} describes $\zp$'s in $G^4$ mapped by $\psi$ in (\ref{G4seq}). Those with a $z_j^p$ factor yield
\begin{eqnarray*}&&\bigoplus_{k=1}^{j-1}y_kTP_{p-1}[y_k]P[y_{k+1}]\bigoplus_{\ell>j}Z_{j+1}^\ell TP_{p-1}[z_\ell]\L_{\ell+1}\\
&=&\bigoplus_{k=1}^{j-1}(P[y_k]-P[y_{k+1}])\ot\L_{j+1}\\
&=&(P[y_1]-P[y_j])\ot\L_{j+1}.\end{eqnarray*}
Combining this with the result of the preceding paragraph yields the desired $P[y_1]\ot\L_{j+1}$.

\medskip
We finish this section by showing that the $\zp$'s including a factor $q$ are obtained exactly once. We omit writing the $q$. The classes which we must obtain are $P[y_1]\bigoplus_{j\ge1}z_j^p\L_{j+1}$.
There are eight ways these appear in $G^i$-sets.
\begin{enumerate}
\item In $G^1$, using (\ref{KkA}) and (\ref{KK}), for $1\le i<j<k$,
$$y_1^{p^{j-1}-1}z_{i,j}z_j^{p-2}\prod_{s=j+1}^{k-1}\{z_s^{p-1},y_s^{p-1}\}\ot TP_{p-1}[y_k]\ot P[y_{k+1}].$$
\item In $G^3$, using (\ref{KkB}) and (\ref{KK}), for $1\le i<k<\ell$,
$$y_1^{p^{k-1}-1}y_kz_{i,k}z_k^{p-2}Z_{k+1}^\ell\ot TP_{p-1}[y_k]\ot P[y_{k+1}]\ot TP_{p-1}[z_\ell]\ot \L_{\ell+1}.$$
\item In $G^3$, using (\ref{KkA}) and (\ref{KK}), for $1\le i<j<k<\ell$,
$$y_1^{p^{j-1}-1}y_kz_{i,j}z_j^{p-2}\prod_{s=j+1}^{k-1}\{z_s^{p-1},y_s^{p-1}\}Z_k^\ell\ot TP_{p-1}[y_k]\ot P[y_{k+1}]\ot TP_{p-1}[z_\ell]\ot \L_{\ell+1}.$$
\item From $\im(\phi')$ in (\ref{G5seq}), for $1\le k<\ell$ and $1\le i\le \ell-k$,
$$y_1^{p^{k-1}-1}z_{i,\ell}\ot TP_{p-1}[y_k]\ot P[y_{k+1}]\ot TP_{p-1}[z_\ell]\ot \L_{\ell+1}.$$
\item From $\psi'$ in (\ref{G5seq}), using (\ref{KkB}) and (\ref{KK}), for $k<\ell$ and $\ell-k<i<\ell$,
$$y_1^{p^{k-1}-1}z_{i,\ell}\ot TP_{p-1}[y_k]\ot P[y_{k+1}]\ot TP_{p-1}[z_\ell]\ot \L_{\ell+1}.$$
\item From $\psi'$ in (\ref{G5seq}), using (\ref{KkA}) and (\ref{KK}), for $i<j<k<\ell$,
$$y_1^{p^{j-1}-1}z_{i,j}z_j^{p-2}\prod_{s=j+1}^{k-1}\{z_s^{p-1},y_s^{p-1}\}\cdot z_\ell\ot TP_{p-1}[y_k]\ot P[y_{k+1}]\ot TP_{p-1}[z_\ell]\ot \L_{\ell+1}.$$
\item From (\ref{78}), using (\ref{KkB}) and (\ref{KK}), for $i<k$ and $1\le e\le p-2$,
$$y_1^{p^{k-1}-1}z_{i,k}z_k^{e-1}P[y_k]\ot\L_{k+1}.$$
\item From (\ref{78}), using (\ref{KkA}) and (\ref{KK}), for $i<j<k$ and $1\le e\le p-2$,
$$y_1^{p^{j-1}-1}z_{i,j}z_j^{p-2}\prod_{s=j+1}^{k-1}\{z_s^{p-1},y_s^{p-1}\}\cdot z_k^eP[y_k]\ot\L_{k+1}.$$
\end{enumerate}

\medskip
First combine (1)+(6) to put a $\ot\L_{k+1}$ at the end of (1), and then,
similarly to the simplification of (\ref{big}), combine with (3)+(8) to get
\begin{equation}\label{S1}\bigoplus_{i<j}y_1^{p^{j-1}-1}P[y_{j+1}]z_{i,j}z_j^{p-2}\L_{j+1}.\end{equation}
We combine and relabel (4)+(5) to give
\begin{equation}\label{S2}\bigoplus_{i<j}y_1^{p^{j-1}-1}TP_{p-1}[y_j]P[y_{j+1}]z_{i,j+1}\L_{j+1}\end{equation}
together with
\begin{equation}\label{S4}\bigoplus_{i\ge j\ge1}y_1^{p^{j-1}-1}TP_{p-1}[y_j]P[y_{j+1}]z_i^p\L_{i+1}.\end{equation}
Let $Y(s)=y_1^{p^s-1}TP_{p-1}[y_{s+1}]P[y_{s+2}]=\langle y_1^i:\nu(i+1)=s\rangle$. Then (\ref{S4}) is
\begin{equation}\label{name}\bigoplus_{i>s\ge0}Y(s)z_i^p\L_{i+1}.\end{equation}
We simplify and relabel (2) to
\begin{equation}\label{S3}\bigoplus_{i<j}y_1^{p^{j-1}-1}y_jTP_{p-1}[y_j]P[y_{j+1}]z_{i,j}z_j^{p-2}\L_{j+1}.\end{equation}
(\ref{S1}), (\ref{S3}), and (7) combine to give
$$\bigoplus_{i<j}y_1^{p^{j-1}-1}P[y_j]z_{i,j}TP_{p-1}[z_j]\L_{j+1}=\bigoplus_{i\le j-1\le t}Y(t)z_{i,j}TP_{p-1}[z_j]\L_{j+1}.$$
For any $t\ge i$, the coefficient of $Y(t)z_i^p$ in (\ref{S2}) plus this is
$$Z_{i+1}^{t+2}\L_{t+2}\oplus\bigoplus_{j=i+1}^{t+1}Z_{i+1}^jTP_{p-1}[z_j]\L_{j+1}=\L_{i+1},$$
as the second part has all monomials not divisible by $Z_{i+1}^{t+2}$. Combining this with (\ref{name}) yields the desired result,
$$\bigoplus_{s\ge0}Y(s)\bigoplus_{i\ge1}z_i^p\L_{i+1}.$$

\section{An explanation of self-duality of $B_k$}\label{optsec}
In this optional section, we discuss some observations about the ASS of $ku^*(K_2)$ and $ku_*(K_2)$ which, among other things, provide an explanation of the self-dual nature of the $B_k$ summands which occur in both $ku^*(K_2)$ and $ku_*(K_2)$. We restrict to $p=2$.

We first observe that, for $k\ge1$, there is an $E_1$-submodule, $\cM_k$, of $H^*(K_2)$ such that $\ext_{E_1}(\zt,\cM_k)$ (resp.~$\ext_{E_1}(\cM_k,\zt)$) is closed under the differentials in the ASS converging to $ku^*(K_2)$ (resp.~$ku_*(K_2)$), yielding the chart $A_k$ (resp.~the $ku$-homology analogue of $A_k$ discussed in Theorem \ref{ku*thm}). For example, with $M_j$ as in (\ref{Mdef}) and $N$ as in Figure \ref{N}, $\cM_3$ is as depicted in Figure \ref{M5pic}.

\bigskip
\begin{minipage}{6in}
\begin{fig}\label{M5pic}

{\bf The $E_1$-module $\cM_3$.}

\begin{center}

\begin{\tz}[scale=.3]

\draw (4,0) -- (6,0);
\draw (8,0) -- (10,0);
\draw (0,0) to[out=45, in=135] (6,0);
\draw (4,0) to[out=315, in=225] (10,0);
\node at (0,-.7) {$17$};
\node at (18,.7) {$26$};
\node at (32,-1) {$33$};
\node at (26,.7) {$30$};
\node at (38,-1) {$36$};
\node at (10,.7) {$22$};
\node at (0,0) {\lb};
\node at (-2,0) {\lb};
\node at (4,0) {\lb};
\node at (6,0) {\lb};
\node at (8,0) {\lb};
\node at (10,0) {\lb};
\node at (16,0) {\lb};
\node at (18,0) {\lb};
\node at (26,0) {\lb};
\node at (28,0) {\lb};
\node at (32,0) {\lb};
\node at (34,0) {\lb};
\node at (36,0) {\lb};
\node at (38,0) {\lb};
\draw (16,0) -- (18,0);
\draw (26,0) -- (28,0);
\draw (32,0) -- (34,0);
\draw (36,0) -- (38,0);
\draw (32,0) to[out=45, in=135] (38,0);
\node at (-2,-2) {$y_1^4$};
\node at (5,-2) {$y_1^3N$};
\node at (17,-2) {$y_1^2M_4$};
\node at (27,-2) {$y_1x_9M_4$};
\node at (34.5,-2) {$M_5$};
\end{\tz}
\end{center}
\end{fig}
\end{minipage}

\bigskip
\ni The two ASSs for $\cM_3$ will yield the charts for $A_3$ and its homology analogue pictured in \cite{DD}.

The situation for $B_k$ is slightly more complicated. There is no $E_1$-submodule of $H^*(K_2)$ which, by itself, can give a chart $B_kz_\ell$. Some of the differentials that truncate $v$-towers in $B_kz_\ell$ come from classes that are part of a summand that includes $y_1^{2^{k-1}-1}qS_{k,\ell}$. We find that, for $2\le k<\ell$, there is an $E_1$-submodule $\cM_{k,\ell}$  of $H^*K_2$ such that $\ext_{E_1}(\zt,\cM_{k,\ell})$ is closed under the differentials in the ASS converging to $ku^*(K_2)$ and yields the chart
$$B_kz_\ell\oplus y_1^{2^{k-1}-1}qS_{k,\ell}\oplus y_kB_kZ_k^\ell.$$
 Note that these three subsets of $ku^*(K_2)$ appeared together in the 10-term exact sequence (\ref{10}).

This $\cM_{k,\ell}$ is symmetric; i.e., there is an integer $D$ such that $\Sigma^D\cM_{k,\ell}^*$ and $\cM_{k,\ell}$ are isomorphic $E_1$-modules, where $\cM_{k,\ell}^*$ is obtained from $\cM_{k,\ell}$ by negating gradings and dualizing $Q_0$ and $Q_1$. This implies that the $v$-towers in $\ext_{E_1}(\zt,\cM_{k,\ell})$ and $\ext_{E_1}(\cM_{k,\ell},\zt)$ correspond nicely. Moreover, the differentials in the two ASSs correspond, too, obtaining isomorphic charts, although the gradings in one  decrease from left to right, while in the other they increase.

We illustrate with an example, $\cM_{3,4}$, and then discuss the implication for self-duality of $B_k$.
In Figure \ref{56}, we depict $\cM_{3,4}$.

\bigskip
\begin{minipage}{6in}
\begin{fig}\label{56}

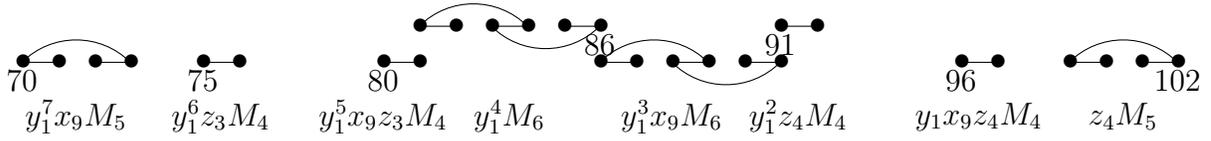
{\bf The $E_1$-module $\cM_{3,4}$.}

\begin{center}

\begin{\tz}[scale=.24]

\draw (0,0) -- (2,0);
\draw (4,0) -- (6,0);
\draw (0,0) to[out=45, in=135] (6,0);
\draw (26,2) to[out=315, in=225] (32,2);
\draw (22,2) to[out=45, in=135] (28,2);
\draw (36,0) to[out=315, in=225] (42,0);
\draw (58,0) to[out=45, in=135] (64,0);
\node at (0,-1) {$70$};
\node at (10,-1) {$75$};
\node at (20,-1) {$80$};
\node at (52,-1) {$96$};
\node at (64,-1) {$102$};
\node at (42,1) {$91$};
\node at (0,0) {\lb};
\node at (2,0) {\lb};
\node at (4,0) {\lb};
\node at (6,0) {\lb};
\node at (12,0) {\lb};
\node at (10,0) {\lb};
\node at (20,0) {\lb};
\node at (22,0) {\lb};
\node at (40,0) {\lb};
\node at (42,0) {\lb};
\node at (32,0) {\lb};
\node at (34,0) {\lb};
\node at (36,0) {\lb};
\node at (38,0) {\lb};
\node at (32,1) {$86$};
\node at (52,0) {\lb};
\node at (54,0) {\lb};
\node at (58,0) {\lb};
\node at (60,0) {\lb};
\node at (62,0) {\lb};
\node at (64,0) {\lb};
\node at (22,2) {\lb};
\node at (24,2) {\lb};
\node at (26,2) {\lb};
\node at (28,2) {\lb};
\node at (30,2) {\lb};
\node at (32,2) {\lb};
\node at (42,2) {\lb};
\node at (44,2) {\lb};
\draw (10,0) -- (12,0);
\draw (20,0) -- (22,0);
\draw (32,0) -- (34,0);
\draw (36,0) -- (38,0);
\draw (32,0) to[out=45, in=135] (38,0);
\node at (3,-3) {$y_1^7x_9M_5$};
\node at (11,-3) {$y_1^6z_3M_4$};
\node at (20,-3) {$y_1^5x_9z_3M_4$};
\node at (27,-3) {$y_1^4M_6$};
\node at (36,-3) {$y_1^3x_9M_6$};
\node at (43,-3) {$y_1^2z_4M_4$};
\node at (53,-3) {$y_1x_9z_4M_4$};
\node at (61,-3) {$z_4M_5$};
\draw (22,2) -- (24,2);
\draw (26,2) -- (28,2);
\draw (30,2) -- (32,2);
\draw (42,2) -- (44,2);
\draw (40,0) -- (42,0);
\draw (52,0) -- (54,0);
\draw (58,0) -- (60,0);
\draw (62,0) -- (64,0);
\end{\tz}
\end{center}
\end{fig}
\end{minipage}

\bigskip

In Figure \ref{MASS}, we depict the ASS chart for both $\ext_{E_1}(\zt,\cM_{3,4})$ and $\ext_{E_1}(\cM_{3,4},\zt)$. They are isomorphic except that, from left to right, the gradings start with 102 for the first and 70 for the second. We label the portions of the chart corresponding to the eight summands of $\cM_{3,4}$ just by the $M$-factor, since accompanying factors differ for the two versions. For example, the $M_5$ on the left-hand side is $z_4M_5$ for the first spectral sequence, and is $y_1^7x_9M_5$ for the second.

\bigskip

\tikzset{
  testpic4/.pic={
\node at (0,0) {\lb};
\node at (2,0) {\lb};
\node at (2,1) {\lb};
\node at (4,2) {\lb};
\node at (6,3) {\lb};
\node at (4,1) {\lb};
\node at (6,2) {\lb};
\node at (8,3) {\lb};
\node at (10,4) {\lb};
\node at (12,5) {\lb};
\node at (14,6) {\lb};
\node at (5,0) {\lb};
\node at (7,1) {\lb};
\node at (10,0) {\lb};
\node at (12,1) {\lb};
\node at (14,2) {\lb};
\node at (16,3) {\lb};
\node at (11,0) {\lb};
\node at (13,1) {\lb};
\node at (15,2) {\lb};
\node at (17,3) {\lb};
\node at (19,4) {\lb};
\node at (21,5) {\lb};
\node at (13,0) {\lb};
\node at (15,1) {\lb};
\node at (17,2) {\lb};
\node at (15,0) {\lb};
\node at (17,1) {\lb};
\node at (16,0) {\lb};
\node at (18,1) {\lb};
\node at (20,2) {\lb};
\node at (22,3) {\lb};
\node at (18,0) {\lb};
\node at (20,1) {\lb};
\node at (22,2) {\lb};
\node at (24,3) {\lb};
\node at (26,4) {\lb};
\node at (28,5) {\lb};
\node at (30,6) {\lb};
\node at (20,0) {\lb};
\node at (22,1) {\lb};
\node at (21,0) {\lb};
\node at (23,1) {\lb};
\node at (26,0) {\lb};
\node at (28,1) {\lb};
\node at (30,2) {\lb};
\node at (32,3) {\lb};
\node at (29,0) {\lb};
\node at (31,1) {\lb};
\node at (31,0) {\lb};
\node at (33,1) {\lb};
\draw [thick] (0,0) -- (2,1) -- (2,0) -- (10,4);
\draw [thick] (10,0) -- (12,1);
\draw [thick] (11,0) -- (17,3);
\draw [thick] (16,0) -- (18,1) -- (18,0) -- (26,4);
\draw [thick] (26,0) -- (28,1);
\draw [color=red] (10,0) -- (10,4);
\draw [color=red] (26,0) -- (26,4);
\draw [dashed] (2,1) -- (6,3);
\draw [dashed] (5,0) -- (7,1);
\draw [dashed] (10,4) -- (16,7);
\draw [dotted] (5,0) -- (4,2);
\draw [dotted] (7,1) -- (6,3);
\draw [dashed] (12,1) -- (16,3);
\draw [dashed] (13,0) -- (17,2);
\draw [dashed] (15,0) -- (17,1);
\draw [dotted] (13,0) -- (12,5);
\draw [dotted] (15,1) -- (14,6);
\draw [dotted] (15,0) -- (14,2);
\draw [dotted] (17,1) -- (16,3);
\draw [dashed] (13,0) -- (13,1);
\draw [dashed] (15,0) -- (15,2);
\draw [dashed] (17,1) -- (17,3);
\draw [dotted] (17,2) -- (16,7);
\draw [dashed] (17,3) -- (21,5);
\draw [dashed] (18,1) -- (22,3);
\draw [dashed] (20,0) -- (22,1);
\draw [dashed] (21,0) -- (23,1);
\draw [dashed] (20,0) -- (20,2);
\draw [dashed] (22,1) -- (22,3);
\draw [dotted] (20,0) -- (19,4);
\draw [dotted] (22,1) -- (21,5);
\draw [dotted] (21,0) -- (20,2);
\draw [dotted] (23,1) -- (22,3);
\draw [dashed] (26,4) -- (30,6);
\draw [dashed] (28,1) -- (32,3);
\draw [dashed] (29,0) -- (31,1);
\draw [dashed] (31,0) -- (33,1);
\draw [dashed] (31,0) -- (31,1);
\draw [dotted] (29,0) -- (28,5);
\draw [dotted] (31,1) -- (30,6);
\draw [dotted] (31,0) -- (30,2);
\draw [dotted] (33,1) -- (32,3);
\node at (0,-.6) {$102$};
\node at (0,-1.2) {$70$};
\node at (2,-.6) {$M_5$};
\node at (6,-.6) {$M_4$};
\node at (9.6,-.6) {$92$};
\node at (9.6,-1.2) {$80$};
\node at (11.2, -.6) {$M_4$};
\node at (14,-.6) {$M_6$};
\node at (16,-.6) {$86$};
\node at (16,-1.2) {$86$};
\node at (19,-.6) {$M_6$};
\node at (22,-.6) {$M_4$};
\node at (25.6,-.6) {$76$};
\node at (25.6,-1.2) {$96$};
\node at (27.2, -.6) {$M_4$};
\node at (30,-.6) {$M_5$};
\draw [color=blue] (-1,0) -- (33,0);
}}
\bigskip
\begin{minipage}{6in}
\begin{fig}\label{MASS}

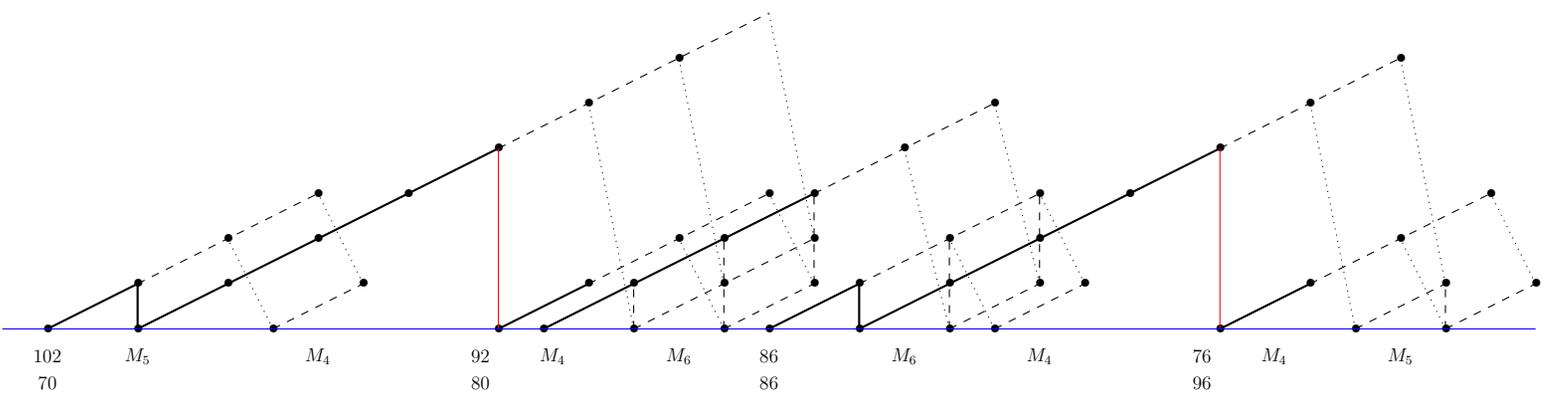
{\bf Two ASSs for $\cM_{2,3}$.}
\begin{center}
\begin{tikzpicture}
  \pic[rotate=90,scale=.6,transform shape] {testpic4};
\end{tikzpicture}
\end{center}
\end{fig}
\end{minipage}

\bigskip

For the $ku^*(K_2)$ version, $B_3z_4$ is on the left hand side of Figure \ref{MASS}, and $y_3B_3z_3$ on the right hand side, with $y_1^3qS_{3,4}$ separating them. The duality isomorphism in Theorem \ref{dual} says that the Pontryagin dual of $B_3z_4$ is isomorphic as a $ku_*$-module to $\Sigma^4$ of the right hand side of the $ku_*(K_2)$ version of Figure \ref{MASS}, and we see that this is isomorphic to a shifted version of $B_3$ with indices negated. This is the self-duality statement, that the Pontryagin dual of $B_k$ is isomorphic as a $ku_*$-module to a shifted version of $B_k$ with indices negated.

\def\line{\rule{.6in}{.6pt}}


\begin{thebibliography}{99}
\bibitem{AH} D.W.Anderson and L.Hodgkin, {\em The $K$-theory of Eilenberg-MacLane complexes}, Topology {\bf 7} (1968) 317--329.
\bibitem{Br} W.Browder, {\em Torsion in $H$-spaces}, Annals of Math {\bf 74} (1961) 24--51.
\bibitem{Cle} A.Cl\'ement, {\em Integral cohomology of finite Postnikov towers}, thesis, University of Lausanne, (2002), available online.
\bibitem{DD} D.M.Davis and J.P.C.Greenlees, {\em Gorenstein duality and universal coefficient theorems},  arXiv.2206.11391.
\bibitem{DW} D.M.Davis and W.S.Wilson,  {\em Stiefel-Whitney classes and immersions of orientable and Spin manifolds},  Topology and Appl, https://doi.org/10.1016/j.topol.2021.107780.
\bibitem{2} \line, {\em The connected $K$-theory of the Eilenberg-MacLane space $K(\zt,2)$}, https://www.lehigh.edu/$\sim$dmd1/kkpaper3.pdf.
\bibitem{DRW} D.M.Davis, D.C.Ravenel, and W.S.Wilson, {\em The connective Morava $K$-theory of the second mod $p$ Eilenberg-MacLane space},  arXiv 2206.14035.
\bibitem{Liu64} A.Liulevicius, {\em Notes on homotopy of Thom spectra}, Amer Jour Math {\bf 86} (1964) 1--16.

\bibitem{Ser} J.P.Serre, {\em Cohomologie modulo 2 des complexes d'Eilenberg MacLane}, Comm Math Helv {\bf 27} (1953) 198--231.
\bibitem{Tam} H.Tamanoi, {\em The image of the $BP$ Thom map for Eilenberg-MacLane spaces}, Trans Amer Math Soc {\bf 349} (1997) 1209--1237.
\bibitem{Wal62} C.T.C.Wall, {\em A characterization of simple modules over the Steenrod algebra mod 2}, Topology {\bf 1} (1962) 249--254.
\bibitem{W} W.S.Wilson, {\em A new relation on the Stiefel-Whitney classes of Spin manifolds}, Ill Jour Math {\bf 17} (1973) 115--127.
\end{thebibliography}
\end{document}